\documentclass[review,10pt]{elsarticle}

\usepackage[english]{babel}
\usepackage{ifpdf}

\usepackage{microtype}
\usepackage[utf8]{inputenc}  
\usepackage{url}
\usepackage{color}
\usepackage{textcomp}
\usepackage{comment}
\usepackage{subfigure}
\usepackage{booktabs}

\usepackage{prettyref}
\usepackage[colorlinks]{hyperref}

\ifpdf
\usepackage{graphicx}
\DeclareGraphicsExtensions{.pdf,.mps,.png,.jpg}
\else
\usepackage{graphicx}
\DeclareGraphicsExtensions{.eps}
\usepackage{pstool}
\fi

\usepackage[super]{nth}
\renewcommand{\th}{\textsuperscript{th}\xspace}

\usepackage{mathrsfs}
\usepackage{wasysym}

\usepackage{amsmath}

\usepackage{amsthm}
\usepackage{amssymb}
\usepackage{mathtools}
\usepackage{dsfont}
\usepackage{upgreek}
\usepackage{lmodern}
\usepackage{nicefrac}

\DeclareMathAlphabet{\mathpzc}{T1}{pzc}{m}{it}

\newtheorem{theorem}{Theorem}
\newtheorem{definition}[theorem]{Definition}
\newtheorem{corollary}[theorem]{Corollary}
\newtheorem{proposition}[theorem]{Proposition}
\newtheorem{remark}[theorem]{Remark}
\newtheorem{lemma}[theorem]{Lemma}

\usepackage{slashed}

\newrefformat{sec}{section~\ref{#1}}
\newrefformat{def}{definition~\ref{#1}}
\newrefformat{th}{theorem~\ref{#1}}
\newrefformat{cor}{corollary~\ref{#1}}
\newrefformat{rem}{remark~\ref{#1}}
\newrefformat{prop}{proposition~\ref{#1}}
\newrefformat{lem}{lemma~\ref{#1}}
\newrefformat{ex}{example~\ref{#1}}
\newrefformat{eq}{equation~\ref{#1}}
\newrefformat{fig}{figure~\ref{#1}}
\newrefformat{tab}{table~\ref{#1}}

\usepackage{xspace}
\usepackage{xargs}

\definecolor{orange}{rgb}{.7,0.4,0.0}
\definecolor{darkr}{rgb}{0.5,0.1,0.1}

\newcommand{\set}[1] {\mathrm{#1}}
\newcommand{\sset}[1] {\mathds{#1}}
\newcommand{\map}[1] {\mathrm{#1}}
\newcommand{\op}[1] {\mathbf{#1}}
\newcommand{\sop}[1] {\mathds{#1}}

\newcommand{\point}[0] {\ \ .}
\renewcommand{\comma}[0] {\ \ ,}
\newcommand{\mt}[1] {\quad\text{#1}\quad}
\newcommand{\transpose}[0] {^{\mathrm{T}}}

\DeclareMathOperator{\supp}{supp}

\newcommand{\knownth}[1] {\textit{#1}}

\newcommand{\dd}[0] {\mathrm{d}}

\newcommand{\E}[0] {\sop{E}}
\newcommand{\p}[0] {\sop{P}}

\newcommand{\hr}[0] {\ensuremath{\mathcal{H}}\xspace}

\newcommand{\Ccb}[0] {\mathcal C_\mathrm{c,b}}

\newcommand{\F}[0] {\mathcal{F}}

\newcommand{\1}[0] {\sop{1}}

\newcommand{\ii}[0] {\mathrm{i}}
\newcommand{\ee}[0] {\mathrm{e}}
\newcommand{\cll}[0] {\!\!\!\!\!\!}
\newcommand{\cl}[0] {\!\!\!\!}

\newcommand{\Cdot}[0] {\raisebox{-0.15ex}{\scalebox{1.25}{$\cdot$}}}
\newcommand{\CDot}[0] {\raisebox{-0.25ex}{\scalebox{1}{$\Cdot$}}}

\renewcommand{\emptyset}{\scalebox{1.25}\o}

\newcommand{\tightoverset}[2]{%
  \mathop{#2}\limits^{\vbox to -.5ex{\kern-0.75ex\hbox{$#1$}\vss}}}
\newcommand{\isomap}[0] {\tightoverset\sim\longrightarrow}
\newcommand*{\longhookrightarrow}{\ensuremath{\lhook\joinrel\longrightarrow}}

\makeatletter
\renewcommand*\env@matrix[1][*\c@MaxMatrixCols c]{%
  \hskip -\arraycolsep
  \let\@ifnextchar\new@ifnextchar
  \array{#1}}
\makeatother

\usepackage{array}
\newcolumntype{x}[1]{>{\centering\hspace{0pt}\arraybackslash}p{#1}}
\newcolumntype{P}[1]{>{\raggedleft\hspace{0pt}\arraybackslash}p{#1}}

\newcommand{\modfrac}[2]{
    \raise.5ex\hbox{\ensuremath{\displaystyle #1}}%
    \kern-.05em/\kern-.05em%
    \lower.25ex\hbox{\ensuremath{\displaystyle #2}}
}

\usepackage[section]{placeins}
\usepackage{cleveref}

\newcommand{\cptp}{\textsc{cptp}\xspace}
\newcommand{\sat}{1-\textsc{in}-3\textsc{sat}\xspace}
\newcommand{\np}{\textsc{NP}\xspace}
\newcommand{\pp}{\textsc{P}\xspace}
\newcommand{\yes}{\textsc{Yes}\xspace}

\usepackage{tikz}
\usetikzlibrary{patterns}
\definecolor{myGreen}{rgb}{0.717647, 0.756863, 0.117647}
\definecolor{myGreenL}{rgb}{0.91634, 0.927959, 0.738562}
\definecolor{myPurple}{rgb}{0.211765, 0.113725, 0.411765}
\definecolor{myRed}{rgb}{0.941176, 0.337255, 0.239216}

\begin{document}
  \title{The Complexity of Divisibility}
  
  \author[cam]{Johannes Bausch\corref{cor}}
  \ead{jkrb2@cam.ac.uk}
  \author[cam,ucl]{Toby Cubitt}
  \ead{t.cubitt@ucl.ac.uk}
  
  \cortext[cor]{Corresponding author}
  \address[cam]{DAMTP, Centre for Mathematical Sciences, University of Cambridge, Wilberforce Road, Cambridge CB3 0WB, UK}
  \address[ucl]{Department of Computer Science, University College London, Gower Street, London WC1E 6BT, UK}
  
  
  \begin{abstract}
    We address two sets of long-standing open questions in linear algebra and probability theory, from a computational complexity perspective: stochastic matrix divisibility, and divisibility and decomposability of probability distributions. We prove that finite divisibility of stochastic matrices is an \np-complete problem, and extend this result to nonnegative matrices, and completely-positive trace-preserving maps, i.e. the quantum analogue of stochastic matrices. We further prove a complexity hierarchy for the divisibility and decomposability of probability distributions, showing that finite distribution divisibility is in \pp, but decomposability is \np-hard. For the former, we give an explicit polynomial-time algorithm. All results on distributions extend to weak-membership formulations, proving that the complexity of these problems is robust to perturbations.
  \end{abstract}
  
  \maketitle

\section{Introduction and Overview}
People have pondered divisibility questions throughout most of western science and philosophy. Perhaps the earliest written mention of divisibility is in \knownth{Aristotle's Physics} in 350BC, in the form of the \knownth{Arrow paradox}---one of \knownth{Zeno of Elea's} paradoxes (ca. 490--430 BC). Aristotle's lengthy discussion of divisibility (he devotes an entire chapter to the topic) was motivated by the same basic question as more modern divisibility problems in mathematics: can the behaviour of an object---physical or mathematical---be subdivided into smaller parts?

For example, given a description of the evolution of a system over some time interval $t$, what can we say about its evolution over the time interval $t/2$? If the system is stochastic, this question finds a precise formulation in the \emph{divisibility problem} for stochastic matrices \cite{Kingman1962}: given a stochastic matrix $\op P$, can we find a stochastic matrix $\op Q$ such that $\op P=\op Q^2$?

This question has many applications. For example, in information theory the stochastic matrices model noisy communication channels, and divisibility becomes important in \knownth{relay coding}, when signals must be transmitted between two parties where direct end-to-end communication is not available \cite{Loring1878}. Another direct use is in the analysis of chronic disease progression \cite{Charitos2008}, where the transition matrix is based on sparse observations of patients, but finer-grained time-resolution is needed. In finance, changes in companies' credit ratings can be modelled using discrete time Markov chains, where rating agencies provide the transition matrix based on annual estimates---for valuation or risk analysis, a transition matrix for a much shorter time periods needs to be inferred \cite{Jarrow1997}.

We can also ask about the evolution of the system for \emph{all} times up to time $t$, i.e.\ whether the system can be described by some continuous evolution. For stochastic matrices, this has a precise formulation in the \emph {embedding problem}: given a stochastic matrix $\op P$, can we find a generator $\op Q$ of a continuous-time Markov process such that $\op P=\exp(\op Qt)$? The embedding problem seems to date back further still, and was already discussed by Elfving in 1937 \cite{Elfving1937}. Again, this problem occurs frequently in the field of systems analysis, and in analysis of experimental time-series snapshots \cite{Cubitt2012, Ljung1987, Nielsen1998}.

Many generalisations of these divisibility problems have been studied in the mathematics and physics literature. For example, the question of square-roots of (entry-wise) nonnegative matrices is an old open problem in matrix analysis \cite{Minc1988}: given an entry-wise nonnegative matrix $\op M$, does it have an entry-wise nonnegative square-root? In quantum mechanics, the analogue of a stochastic matrix is a completely-positive trace preserving (CPTP) map, and the corresponding divisibility problem asks: when can a CPTP map $\op T$ be decomposed as $\op T = \op R\circ\op R$, where $\op R$ is itself CPTP? The continuous version of this, whether a CPTP can be embedded into a completely-positive semi-group, is sometimes called the \emph{Markovianity problem} in physics \cite{Cubitt2012a}---the latter again has applications to subdivision coding of quantum channels in quantum information theory \cite{Muller-Hermes2015}.

Instead of dynamics, we can also ask whether the description of the static state of a system can be subdivided into smaller, simpler parts. Once again, probability theory provides a rich source of such problems. The most basic of these is the classic topic of divisibile distributions: given a random variable $X$, can it be decomposed into $X = Y + Z$ where $Y,Z$ are some other random variables? What if $Y$ and $Z$ are identically distributed? If we instead ask for a decomposition into infinitely random variables, this becomes the question of whether a distribution is infinitely divisible.

In this work, we address two of the most long-standing open problems on divisibility: divisibility of stochastic matrices, and divisibility and decomposability of probability distributions. We also extend our results to divisibility of nonnegative matrices and completely positive maps. Surprisingly little is known about the divisibility of stochastic matrices. Dating back to 1962 \cite{Kingman1962}, the most complete characterization remains for the case of a $2\times2$ stochastic matrix \cite{He2003}. The infinite divisibility problem has recently been solved \cite{Cubitt2012a}, but the finite case remains an open problem. Divisibility of random variables, on the other hand, is a widely-studied topic. Yet, despite first results dating back as far as 1934 \cite{Cochran1934}, no general method of answering whether a random variable can be written as the sum of two or more random variables---whether distributed identically, or differently---was known.

We focus on the computational complexity of these divisibility problems. In each case, we show which of the divisibility problems have efficient solutions---for these, we give an explicit efficient algorithm. For all other cases, we prove reductions to the famous  $\pp=\np$-conjecture, showing that those problems are \np-hard. This essentially implies that---unless  $\pp=\np$---the geometry of the corresponding divisible and non-divisible is highly complex, and these sets have no simple characterisation beyond explicit enumeration. In particular, this shows that any future concrete classification of these \np-hard problems will be at least as hard as answering $\pp=\np$.

The following theorems summarize our main results on maps. Precise formulations and proofs can be found in \cref{sec:cptpstochastic}.
\begin{theorem}\label{th:stochhard}
	Given a stochastic matrix $\op P$, deciding whether there exists a stochastic matrix $\op Q$ such that $\op P=\op Q^2$ is \np-complete.
\end{theorem}
\begin{theorem}\label{th:cptphard}
	Given a \cptp map $\op B$, deciding whether there exists a \cptp map $\op A$ such that $\op B=\op A\circ\op A$ is \np-complete.
\end{theorem}

In fact, the last two theorems are strengthenings of the following result.
\begin{theorem}\label{th:nonneghard}
	Given a nonnegative matrix $\op M$, deciding whether there exists a nonnegative matrix $\op N$ such that $\op M=\op N^2$ is \np-complete.
\end{theorem}


The following theorems summarize our main results on distributions. Precise formulations and proofs can be found in \cref{sec:distributions}.
\begin{theorem}\label{th:div}
	Let $X$ be a finite discrete random variable. Deciding whether $X$ is $n$-divisible---i.e. whether there exists a random variable $Y$ such that $X = \sum_{i=1}^n Y$---is in \pp.
\end{theorem}
\begin{theorem}\label{th:divweak}
	Let $X$ be a finite discrete random variable, and $\epsilon>0$. Deciding whether there exists a random variable $Y$ $\epsilon$-close to $X$ such that $Y$ is $n$-divisible, or that there exists such a $Y$ that is nondivisible, is in \pp.
\end{theorem}
\begin{theorem}\label{th:dec}
	Let $X$ be a finite discrete random variable. Deciding whether $X$ is decomposable---i.e. whether there exist random variables $Y,Z$ such that $X = Y + Z$---is \np-complete.
\end{theorem}
\begin{theorem}\label{th:decweak}
	Let $X$ be a finite discrete random variable, and $\epsilon>0$. Deciding whether there exists a random variable $Y$ $\epsilon$-close to $X$ such that $Y$ is decomposable, or that there exists such a $Y$ that is indecomposable, is \np-complete.
\end{theorem}

It is interesting to contrast the results on maps and distributions. In the case of maps, the homogeneous $2$-divisibility problems are already \np-hard, whereas finding an inhomogeneous decomposition is straightforward. For distributions, on the other hand, the homogeneous divisibility problems are efficiently solvable to all orders, but becomes \np-hard if we relax it to the inhomogeneous decomposibility problem.

This difference is even more pronounced for infinite divisibility. The infinite divisibility problem for maps is \np-hard (shown in \cite{Cubitt2012a}), whereas the infinite divisibility and decomposibility problems for distributions are computationally trivial, since indivisible and indecomposible distributions are both dense---see \prettyref{sec:contdec} and \ref{sec:contdiv}.

The paper is divided into two parts. We first address stochastic matrix and \cptp divisibility in \prettyref{sec:cptpstochastic}, obtaining results on entry-wise positive matrix roots along the way. Divisibility and decomposability of probability distributions is addressed in \prettyref{sec:distributions}. In both sections, we first give an overview of the history of the problem, stating previous results and giving precise definitions of the problems. We introduce the necessary notation at the beginning of each section, so that each section is largely self-contained.

\section{CPTP and Stochastic Matrix Divisibility}\label{sec:cptpstochastic}
\subsection{Introduction}
Mathematically, subdividing Markov chains is known as the \knownth{finite divisibility} problem. The simplest case is the question of finding a stochastic root of the transition matrix (or a \cptp root of a \cptp map in the quantum setting), which corresponds to asking for the evolution over half of the time interval.
%
%
While the question of divisibility is rather simple to state mathematically, it is not clear a priori whether a stochastic matrix root for a given stochastic matrix exists at all. Historically, this has been a long-standing open question, dating back to at least 1962 \cite{Kingman1962}. Matrix roots were also suggested early on in other fields, such as economics and general trade theory, at least as far back as 1967 \cite{Waugh1967}, to model businesses and the flow of goods. Despite this long history, very little is known about the existence of stochastic roots of stochastic matrices. The most complete result to date is a full characterization of $2\times2$ matrices, as given for example in \cite{He2003}. The authors mention that \emph{``\dots it is quite possible that we have to deal with the stochastic root problem on a case-by-case basis.''} This already suggests that there might not be a simple mathematical characterisation of divisible stochastic matrices---meaning one that is simpler than enumerating the exponentially many roots and checking each one for stochasticity.

There are similarly few results if we relax the conditions on the matrix normalization slightly, and ask for (entry-wise) nonnegative roots of (entry-wise) nonnegative matrices---for a precise formulation, see \prettyref{def:nonneg} and \ref{def:stoch}. An extensive overview can be found in \cite{Minc1988}.
%
Following this long history of classical results, quantum channel divisibility recently gained attention in the quantum information literature. The foundations were laid in \cite{Wolf2008}, where the authors first introduced the notion of \knownth{channel divisibility}. A divisible quantum channel is a \cptp map that can be written as a nontrivial concatenation of two or more quantum channels.

A related question is to ask for the evolution under infinitesimal time steps, which is equivalent to existence of a logarithm of a stochastic matrix (or \cptp map) that generates a stochastic (resp.\ \cptp) semi-group. Classically, the question is known as \knownth{Elfving's} problem or the \knownth{embedding problem}, and seems to date back even further than the finite case to 1937 \cite{Elfving1937}. In the language of Markov chains, this corresponds to determining whether a given stochastic matrix can be embedded into an underlying continuous time Markov chain. Analogously, infinite quantum channel divisibility---also known as the \knownth{Markovianity} condition for a \cptp map---asks whether the dynamics of the quantum system can be described by a \knownth{Lindblad} master equation \cite{Lindblad1976,Gorini1976}. The infinite divisibility problems in both the classical and quantum case were recently shown to be \np-hard \cite{Cubitt2012a}. Formulated as weak membership problems, these results imply that it is \np-hard to extract dynamics from experimental data \cite{Cubitt2012}.

However, while related, it is not at all clear that there exists a reduction of the finite divisibility question to the case of infinite divisibility. In fact, mathematically, the infinite divisibility case is a special case of finite divisibility, as a stochastic matrix is infinitely divisible if and only if it admits an $n$\th root for all $n\in\sset N$ \cite{Kingman1962}.

The finite divisibility problem for stochastic matrices is still an open question, as are the nonnegative matrix and \cptp map divisibility problems. We will show that the question of existence of stochastic roots of a stochastic matrix is \np-hard. We also extend this result to (doubly) stochastic matrices, nonnegative matrices, and \cptp maps. 

We start out by introducing the machinery we will use to prove \prettyref{th:nonneghard} and \ref{th:stochhard} in \prettyref{sec:prelims1}. A reduction from the quantum to the classical case can be found in \prettyref{sec:cptp}, from the nonnegative to the stochastic case in \prettyref{sec:root} and the main result---in a mathematically rigorous formulation---is then presented as \prettyref{th:reduction} in \prettyref{sec:sat}.

\subsection{Preliminaries}\label{sec:prelims1}
\subsubsection{Roots of Matrices}
In our study of matrix roots we restrict ourselves to the case of square roots. The more general case of $p$\th roots of matrices remains to be discussed. We will refer to square roots simply as roots. To be explicit, we state the following definition.
\begin{definition}\label{def:roots}
	Let $\op M\in\sset K^{d\times d}$, $d\in\sset N$, $\sset K$ some field. Then we say that $\op R\in\sset K^{d\times d}$ is a root of $\op M$ if $\op R^2=\op M$. We denote the set of all roots of $\op M$ with $\sqrt{\op M}$.
\end{definition}

Following the theory of matrix functions---see for example \cite{Higham1987}---we remark that in the case of nonsingular $\op M$, $\sqrt{\op M}$ is nonempty and can be expressed in Jordan normal form via $\sqrt{\op M}=\op Z\op J\op Z^{-1}$ for some invertible $\op Z$, where $\op J=\map{diag}(\op J_1^\pm,\ldots,\op J_m^\pm)$. Here $\op J_i^\pm$ denotes the $\pm$-branch of the root function $f(x)=\sqrt x$ of the Jordan block corresponding to the $i$\th eigenvalue $\lambda_i$,
\[
\op J_i^\pm=\begin{pmatrix}
\pm f(\lambda_i) & \pm f'(\lambda_i)/1! & \ldots & \pm f^{(m_i-1)}(\lambda_i)/(m_i-1)! \\
0 & \pm f(\lambda_i) & \ddots & \vdots \\
\vdots & \ddots & \ddots & \pm f'(\lambda_i)/1! \\
0 & \hdots & 0 & \pm f(\lambda_i)
\end{pmatrix}\point
\]
If $\op M$ is diagonalisable, $\op J$ simply reduces to the canonical diagonal form $\op J=\map{diag}(\pm\sqrt{\lambda_1},\ldots,\pm\sqrt{\lambda_m})$.

If $\op M$ is derogatory---i.e.\ there exist multiple Jordan blocks sharing the same eigenvalue $\lambda$---it has continuous families of so-called \emph{nonprimary} roots $\sqrt{\op M}=\op Z\op U\op J\op U^{-1}\op Z^{-1}$, where $\op U$ is an arbitrary nonsingular matrix that commutes with the Jordan normal form $[\op U,\op J]=0$.

We cite the following result from \cite[Th.~2.6]{Higham2011}.
\begin{theorem}[Classification of roots]\label{th:roots}
	Let $\op M\in\sset K^{d\times d}$ have the Jordan canonical form $\op Z\op\Lambda\op Z^{-1}$, where $\op\Lambda=\map{diag}(\op J_0,\op J_1)$, such that $\op J_0$ collects all Jordan blocks corresponding to the eigenvalue $0$, and $\op J_1$ collects the remaining ones. Assume further that
	\[d_i:=\dim(\map{ker}\,\op M^i)-\dim(\map{ker}\,\op M^{i-1})\]
	has the property that for all $i\in\sset N_{\ge0}$, no more than one element of the sequence satisfies $d_i\in(2i,2(i+1))$.
	Then $\sqrt{\op M}=\op Z\sqrt{\op\Lambda}\op Z^{-1}$, where $\sqrt{\op\Lambda}=\map{diag}(\sqrt{\op J_0},\sqrt{\op J_1})$.
\end{theorem}

For a given matrix, the classification gives the set of all roots. If $\op M$ is a real matrix, a similar theorem holds and there exist various numerical algorithms for calculating real square roots, see for example \cite{Higham1987}.

\subsubsection{Roots of Stochastic Matrices}\label{sec:smatroot}
Remember the following two definitions.
\begin{definition}\label{def:nonneg}
	A matrix $\op M\in\sset K^{d\times d}$ is said to be nonnegative if $0\le\op M_{ij}\ \forall i,j=0,\ldots,d$.
\end{definition}
\begin{definition}\label{def:stoch}
	A matrix $\op Q\in\sset K^{d\times d}$ is said to be stochastic if it is nonnegative and $\sum_{k=1}^d\op Q_{ik}=1\ \forall i=0,\ldots,d$.
\end{definition}

In contrast to finding a general root of a matrix, very little is known about the existence of nonnegative roots of nonnegative matrices---or stochastic roots of stochastic matrices---if $d\ge3$. For stochastic matrices and in the case $d=2$, a complete characterization can be given explicitly, and for $d\ge3$, all real stochastic roots that are functions of the original matrix are known, as demonstrated in \cite{He2003}. Further special classes of matrices for which a definite answer exists can be found in \cite{Higham2011}. But even for $d=3$, the general case is still an open question---see \cite[ch. 2.3]{Lin2011} for details.

Indeed, a stochastic matrix may have no stochastic root, a primary or nonprimary root---or both. To make things worse, if a matrix has a $p$\th stochastic root, it might or might not have a $q$\th stochastic root if $p\nmid q$---$p$ is not a divisor of $q$---, $q>p$ or $q\nmid p$, $q<p$.

A related open problem is the inverse eigenspectrum problem, as described in the extensive overview in \cite{Egleston2004}. While the sets $\Omega_n\subset\sset D$---denoting all the possible eigenvalues of an $n$-dimensional nonnegative matrix---can be given explicitly, and hence also $\Omega_n^p$, almost nothing is known about the sets for the entire eigenspectrum. Any progress in this area might yield necessary conditions for the existence of stochastic roots.

In recent years, some approaches have been developed to approximate stochastic roots numerically, see the comments in \cite[sec. 4]{He2003}. Unfortunately, most algorithms are highly unstable and do not necessarily converge to a stochastic root. A direct method using nonlinear optimization techniques is difficult and depends heavily on the algorithm employed \cite{Lin2011}.

It remains an open question whether there exists an efficient algorithm that decides whether a stochastic matrix $\op Q$ has a stochastic root.

\begin{em}
	In this paper, we will prove that this question is \np-hard to answer.
\end{em}

\subsubsection{The Choi Isomorphism}
For the results on \cptp maps, we will need the following basic definition and results.
\begin{definition}\label{def:cptp}
	Let $\op A:\hr\longrightarrow\hr$ be a linear map on $\hr=\sset C^{d\times d}$. We say that $\op A$ is positive if for all Hermitian and positive definite $\op \rho\in\hr$, $\op A\op\rho$ is Hermitian and positive definite. It is said to be completely positive if $\op A\otimes\1_n$ is positive $\forall n\in\sset N$.

	A map $\op A$ which is completely positive and trace-preserving---i.e. $\map{tr}(\op A\rho)=\map{tr}\ \rho\ \forall \rho\in\hr$---is called a \emph{completely positive trace-preserving} map, or short \cptp map.
\end{definition}

In contrast to positivity, complete positivity is easily characterized using the well-known \knownth{Choi-Jamiolkowski} isomorphism---cf.\ \cite[Th.~2]{Choi1975}.
\begin{remark}\label{rem:choi}
	Let the notation be as in \prettyref{def:cptp} and pick a basis $e_1,\ldots,e_d$ of $\sset C^d$. Then $\op A$ is completely positive if and only if the \emph{Choi} matrix
	\[
	\map C_\op A:=(\1_d\otimes\op A)\Omega\Omega\transpose=\sum_{i,j=1}^d e_ie_j\transpose\otimes\op A(e_ie_j\transpose)
	\]
	is positive semidefinite, where $\Omega:=\sum_{i=1}^d e_i\otimes e_i$.
\end{remark}

The condition of trace-preservation then translates to the following.
\begin{remark}
	A map $\op A$ is trace-preserving if and only if $\map{tr_2}(\map C_\op A)=\1_d$, where $\map{tr_2}$ denotes the partial trace over the second pair of indices.
\end{remark}

\subsection{Equivalence of Computational Questions}\label{sec:equiv}
In the following we denote with $S$ some arbitrary finite index set, not necessarily the same for all problems. We begin by defining the following  decision problems.

\begin{definition}[\textsc{\cptp Divisibility}]\leavevmode\label{def:cptpdiv}
\begin{description}
\item[Instance.] \cptp map $\op B\in \sset Q^{d\times d}$.
\item[Question.] Does there exist a \cptp map $\op A:\op A^2=\op B$?
\end{description}
\end{definition}

\begin{definition}[\textsc{\cptp Root}]\leavevmode
\begin{description}
\item[Instance.] Family of matrices $(\op A_s)_{s\in\set S}$ that comprises all the roots of a matrix $\op B$.
\item[Question.] Does there exist an $s\in S:\op A_s$ is a \cptp map?
\end{description}
\end{definition}

\begin{definition}[\textsc{Stochastic Divisibility}]\leavevmode
\begin{description}
\item[Instance.] Stochastic matrix $\op P\in\sset Q^{d\times d}$.
\item[Question.] Does there exist a stochastic matrix $\op Q: \op Q^2=\op P$?
\end{description}
\end{definition}

\begin{definition}[\textsc{Stochastic Root}]\leavevmode
\begin{description}
\item[Instance.] Family of matrices $(\op Q_s)_{s\in\set S}$ comprising all the roots of a matrix $\op P$.
\item[Question.] Does there exist an $s\in\set S: \op Q_s$ stochastic?
\end{description}
\end{definition}

\begin{definition}[\textsc{Nonnegative Root}]\leavevmode\label{def:nonnegprob}
\begin{description}
\item[Instance.] Family of matrices $(\op M_s)_{s\in\set S}$ comprising all the roots of a matrix $\op N$, where all $\op M_s$ have at least one positive entry.
\item[Question.] Does there exist an $s\in\set S: \op M_s$ nonnegative?
\end{description}
\end{definition}

\begin{figure}
\centering
\begin{tikzpicture}[node distance=2.38cm,auto,every node/.style={scale=0.9}]
  \node (sat) [fill=myGreenL] {\sat};
  \node (nnroot) [above of=sat,node distance=1.42cm] {\textsc{\shortstack{Nonnegative\\ Root}}};
  \node (nndiv) [left of=nnroot,node distance=3.92cm] {\textsc{\shortstack{Nonnegative\\ Divisibility}}};
  \node (sroot) [above left of=nnroot] {\textsc{\shortstack{Stochastic\\ Root}}};
  \node (sdiv) [above left of=sroot,node distance=2.18cm] {\textsc{\shortstack{Stochastic\\ Divisibility}}};
  \node (cproot) [above right of=sroot,node distance=2.18cm] {\textsc{\shortstack{\cptp\\ Root}}};
  \node (cpdiv) [right of=cproot,node distance=4.58cm] {\textsc{\shortstack{\cptp\\ Divisibility}}};
  \node (dsroot) [right of=nnroot,node distance=4.5cm] {\textsc{\shortstack{Doubly Stochastic\\ Root}}};
  \node (dsdiv) [above right of=dsroot,node distance=2.18cm] {\textsc{\shortstack{Doubly Stochastic\\ Divisibility}}};
  \draw[->,swap] (sat) to node {\prettyref{sec:sat}} (nnroot);
  \draw[->,swap] (nnroot) to node[pos=.45] {\prettyref{sec:root}} (sroot);
  \draw[->,dotted] (sroot) [bend right=20] to node {} (nnroot);
  \draw[->] (sroot) to node {} (sdiv);
  \draw[->,dashed] (sdiv) [bend right=20] to node {} (sroot);
  \draw[->] (nnroot) to node {} (dsroot);
  \draw[->] (sroot) to node[swap,pos=.7] {\prettyref{sec:cptp}} (cproot);
  \draw[->] (cproot) to node[swap] {} (cpdiv);
  \draw[->,dashed,transform canvas={yshift=2.6mm}] (cpdiv.west) [bend right=12] to node {} (cproot.east);
  \draw[->] (dsroot) to node {} (dsdiv);
  \draw[->,dashed] (dsdiv) [bend right=20] to node {} (dsroot);
  \draw[->] (nnroot) to node {} (nndiv);
  \draw[->,dashed] (nndiv) [bend right=12] to node {} (nnroot);
\end{tikzpicture}
\caption{Complete chain of reduction for our programs. The dashed line between the \textsc{Divisibility} and \textsc{Root} problems hold for non-derogatory matrices, respectively. The dotted line between \textsc{Stochastic Root} and \textsc{Nonnegative Root} holds only for irreducible matrices. The doubly stochastic and nonnegative branch are included for completeness but not described in detail here---see \prettyref{cor:doubly}.}
\label{fig:reduction}
\end{figure}
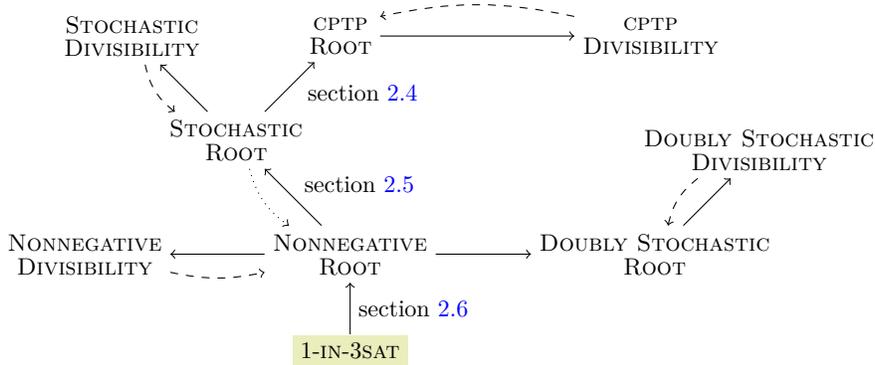

\begin{theorem}\label{th:reduction}
The reductions as shown in \prettyref{fig:reduction} hold.
\end{theorem}

\begin{proof}
The implication \textsc{Stochastic Divisibility}$\longleftarrow$\textsc{Stochastic Root} needs one intermediate step. If $\op P$ is not stochastic, the answer is negative. If it is stochastic, we can apply \textsc{Stochastic Divisibility}.
The opposite direction holds for non-derogatory stochastic $\op P$: in this case we can enumerate all roots of $\op P$ as a finite family which forms a valid instance for \textsc{Stochastic Root}.

The reduction \textsc{Stochastic Root}$\longleftarrow$\textsc{Nonnegative Root} can be resolved by \prettyref{lem:nonnegequiv1} and \prettyref{lem:nonnegequiv2}---we construct a family of matrices $(\op Q_s)_{s\in\set S}$ that contains a stochastic root iff $(\op M_s)_{s\in\set S}$ contains a nonnegative root. The result then follows from applying \textsc{Stochastic Root}. If our stochastic matrix $\op P$ is irreducible, then any nonnegative root $\op Q_{s'}:\op Q_{s'}^2=\op P$ is stochastic, and in that case \textsc{Stochastic Root}$\longleftrightarrow$\textsc{Nonnegative Root}---see \cite[sec. 3]{Higham2011} for details.

The link \textsc{\cptp Divisibility}$\longleftarrow$\textsc{\cptp Root} again needs the following intermediate step. If $\op A$ is not \cptp, the answer is negative. If it is \cptp, then we can apply \textsc{\cptp Divisibility}. Similarly, if $\op A$ is non-derogatory, the reduction works in the opposite direction as well.

The direction \textsc{\cptp Root}$\longleftarrow$\textsc{Stochastic Root} follows from \prettyref{cor:cptpequiv1}. We start out with a family $(\op Q_s)_{s\in S}$ comprising all the roots of a stochastic matrix $\op P$. Then let $(\op A_s:=\map{emb}\ \op Q_s)_{s\in S}$---this family then comprises all of the roots of $\op B:=\op A^2_k\equiv\op A^2_s\ \forall k,s$.
Furthermore, by \prettyref{lem:cptp}, there exists a \cptp $\op A_s$ if and only if there exists a stochastic $\op Q_s$, and the reduction follows.

Finally, we can extend our reduction to the programs \textsc{Doubly Stochastic Root} and \textsc{Doubly Stochastic Divisibility} as well as \textsc{Nonnegative Divisibility}, defined analogously, see our comment in \prettyref{cor:doubly} and the complete reduction tree in \prettyref{fig:reduction}.
\end{proof}

At this point, we observe the following fact.
\begin{lemma}\label{lem:npcontainment}
All the above \textsc{Divisibility} and \textsc{Root} problems in \prettyref{def:cptpdiv} to \ref{def:nonnegprob} are contained in \np.
\end{lemma}
\begin{proof}
It is straightforward to come up with a witness and a verifier circuit that satisfies the definition of the decision class \np. For example in the \cptp case, a witness is a matrix root that can be checked to be a \cptp map using \prettyref{rem:choi} and squared in polynomial time, which is the verifier circuit. Both circuit and witness are clearly poly-sized and hence the claim follows.
\end{proof}

By encoding an instance of \sat into a family of nonnegative matrices $(\op M_s)_{s\in\set S}$, we show the implication \sat\!\!$\longrightarrow$\textsc{Nonnegative Root} and \sat\!\!$\longrightarrow$\textsc{(Doubly) Stochastic/\cptp Divisibility}, accordingly, from which \np-hardness of \textsc{(Doubly) Stochastic/\cptp Divisibility} follows. The entire chain of reduction can be seen in \prettyref{fig:reduction}.

\subsection{Reduction of \textsc{Stochastic Root} to \textsc{CPTP Root}}\label{sec:cptp}
This reduction is based on the following embedding.
\begin{definition}
Let $\{e_i\}$ be an orthonormal basis of $\sset K^d$. The embedding $\map{emb}$ is defined as
\begin{align*}
\map{emb}:\sset K^{d\times d}&\longhookrightarrow\sset K^{d^2\times d^2}\comma\\
\op A&\longmapsto\op B:=\sum_{i,j=1}^d\op A_{ij}(e_i\otimes e_i)(e_j\otimes e_j)\transpose=\sum_{i,j=1}^d\op A_{ij}(e_ie_j\transpose)\otimes(e_ie_j\transpose)\point
\end{align*}
\end{definition}
We observe the following.
\begin{lemma}\label{lem:cptp}
We use the same notation as in \prettyref{rem:choi}. Let $\op A\in\sset K^{d\times d}$ and $\op B:=\map{emb}\ \op A$. Then $\op A$ is positive (nonnegative) if and only if the Choi matrix  $\map C_\op B$ is positive (semi-)definite. Furthermore, the row sums of $\op A$ are $1$---i.e. $\sum_{j=1}^d\op A_{ij}=1\ \forall j=1,\ldots,d$---if and only if $\map{tr_2}(\map C_\op B)=\1_d$. In addition, the spectrum of $\op B$ satisfies $\sigma(\op B)\subseteq\sigma(\op A)\cup\{0\}$.
\end{lemma}
\begin{proof}
The first claim follows directly from the matrix representation of our operators. There, the \knownth{Choi} isomorphism is manifest as the reshuffling operation or partial transpose
\[
\cdot^\Gamma:\sset K^{d^2\times d^2}\longrightarrow\sset K^{d^2\times d^2},
\left[(e_ie_j\transpose)\otimes(e_ie_j\transpose)\right]^\Gamma\longmapsto(e_ie_i\transpose)\otimes(e_je_j\transpose)\point
\]
For more details, see e.g. \cite{Bengtsson2006}.

The second statement follows from
\begin{align*}
\map{tr_2}(\map C_\op B)&=\map{tr_2}\left(\sum_{i,j=1}^d\op A_{ij}(e_ie_j\transpose)\otimes(e_ie_j\transpose)\right)\\
&=\sum_{i,j=1}^d\op A_{ij}e_ie_i\transpose=\map{diag}\left(\sum_{j=1}^d\op A_{1j},\ldots,\sum_{j=1}^d\op A_{dj}\right)\point
\end{align*}
The final claim is trivial.
\end{proof}

This remark immediately yields the following consequence.
\begin{corollary}\label{cor:cptpequiv1}
For a family of stochastic matrices $(\op Q_s)_{s\in S}$ parametrized by the index set $S$, there exists a family of square matrices $(\op A_s)_{s\in S}:=(\map{emb}\ \op Q_s)_{s\in\set S}$, such that $(\op Q_s)_{s\in\set S}$ contains a stochastic matrix if and only if $(\op A_s)_{s\in\set S}$ contains a \cptp matrix.
\end{corollary}

\subsection{Reduction of \textsc{Nonnegative Root} to \textsc{Stochastic Root}}\label{sec:root}
The difference between \textsc{Nonnegative Root} and \textsc{Stochastic Root} is the extra normalization condition in the latter, see \prettyref{def:stoch}.
The following two lemmas show that this normalization does not pose an issue, so we can efficiently reduce the problem \textsc{Nonnegative Root} to \textsc{Stochastic Root}.
\begin{lemma}\label{lem:nonnegequiv1}
For a family of square matrices $(\op M_s)_{s\in\set S}$ parametrized by the index set $\set S$, all of which with at least one positive entry, there exists a family of square matrices $(\op Q_s)_{s\in\set  S}$ such that $(\op M_s)_{s\in\set S}$ contains a nonnegative matrix if and only if $(\op Q_s)_{s\in\set S}$ contains a stochastic matrix and such that $\map{rank}\,\op Q_s=\map{rank}\,\op M_s+2\ \forall s\in\set S$. Furthermore, $(\op Q_s)_{s\in\set S}$ can be constructed efficiently from $(\op M_s)_{s\in\set S}$.
\end{lemma}

\begin{proof}
We explicitly construct our family $(\op Q_s)_{s\in\set S}$ as follows.
Pick an $s\in\set S$ and denote $\op M:=\op M_s$. Let $d$ be the dimension of $\op M$. We first pick $a\in\sset R^+$ such that $a\max_{ij}\op M_{ij}=1/2$\footnote{The exact bound is $43/81$.} and define
\begin{align*}
\op Q_s:\!&=\frac1{1764d}\begin{pmatrix}
 1764 a\op M+637 & 735-1260 a\op M & 392-504 a\op M \\
 735-1260 a\op M & 900 a\op M+1029 & 360 a\op M \\
 392-504 a\op M & 360 a\op M & 144 a\op M+1372
\end{pmatrix}\\
&\equiv\frac{a}{d}AA\transpose\otimes\op M+\frac1d\left(BB\transpose+CC\transpose\right)\otimes\op 1\comma
\end{align*}
where by sum of matrix $\op M$ and scalar $x$ we mean $\op M+x\op 1$, $\op 1:=(1)_{1\le i,j\le d}\in\sset R^{d\times d}$, and
\[
A:=\left(1,-\frac57,-\frac27\right)\transpose\comma\quad
B:=\left(\frac16,\frac12,-\frac23\right)\transpose\comma\quad
C:=-\frac1{\sqrt3}\left(1,1,1\right)\transpose\point
\]
Observe that $\{A,B,C\}$ form an orthogonal set---if one wishes, normalizing and pulling out the constant as eigenvalue to the corresponding eigenprojectors would work equally well.

By construction, $\op Q_s$ is nonnegative if and only if $\op M_s$ is. Since the row sums of $\op Q_s$ are always $1$, $\op Q_s$ is stochastic if and only if $\op M_s$ is nonnegative, and the claim follows.
\end{proof}

\begin{lemma}\label{lem:nonnegequiv2}
Let the notation be as in \prettyref{lem:nonnegequiv1} and write $\sqrt{\op N}$ for the set of roots of $\op N$, see \prettyref{def:roots}. Assume $(\op M_s)_{s\in\set S}=\sqrt{\op N}$ for some $\op N\in\sset C^{d\times d}$. Then there exists a $\op P\in\sset C^{d\times d}$, such that $\op Q_s^2=\op P\ \forall s\in S$ and $(\op Q_s)_{s\in\set S}\subset\sqrt{\op P}$. Furthermore, the complement of $(\op Q_s)_{s\in\set S}$ in $\sqrt{(\op P}$ does not contain any stochastic roots.
\end{lemma}
\begin{proof}
The first statement is obvious, since for all $s\in\set S$,
\begin{align*}
\op Q_s^2&=\frac{a^2}{d^2}\frac{78}{49}AA\transpose\otimes\op M_s^2+\frac1d\left(\frac{13}{18}BB\transpose+CC\transpose\right)\otimes\op 1=:\op P\comma
\end{align*}
and hence clearly $(\op Q_s)_{s\in\set S}\subset\sqrt{\op P}$.

The last statement is not quite as straightforward---it is the main reason our carefully crafted matrix $\op Q_s$ has its slightly unusual shape. All possible roots of $\op P$ are of the form
\[
\sqrt{\op P}=
\frac adAA\transpose\otimes\sqrt{\op N}\pm\frac1d\left(BB\transpose\pm CC\transpose\right)\otimes\op 1\point
\]
It is easy to check that none of the other sign choices yields any stochastic matrix, so the claim follows.\footnote{The reader may try to find a simpler matrix that does the trick.}
\end{proof}

\begin{corollary}\label{cor:doubly}
The results of \prettyref{lem:nonnegequiv1} and \ref{lem:nonnegequiv2} also hold for doubly stochastic matrices---observe that our construction of $\op Q_s$ is already doubly stochastic.
\end{corollary}

\subsection{Reduction of \sat to \textsc{Nonnegative Root}}\label{sec:sat}
We now embed an instance of a boolean satisfiability problem, \sat---see \prettyref{def:sat} for details---into a family of matrices $(\op M_s)_{s\in\set S}$ in a way that there exists an $s$ such that $\op M_s$ is nonnegative if and only if the instance of \sat is satisfiable. The construction is inspired by the one in \cite{Cubitt2012a}.

We identify
\begin{equation}\label{eq:bool}
{\textsc{true}}\longleftrightarrow1\comma\quad\mathrm{\textsc{false}}\longleftrightarrow-1\point
\end{equation}
Denote with $(m_{i1},m_{i2},m_{i3})\in\{\pm1\}^3$ the three boolean variables occurring in the $i$\th boolean clause, and let $m_i\in\{\pm1\}$ stand for the single $i$\th boolean variable.
Then \sat translates to the inequalities
\begin{equation}\label{eq:3sat2}
-\frac32\le m_{i1}+m_{i2}+m_{i3}\le-\frac12\quad\forall i=1,\ldots,n_c\point
\end{equation}

\begin{theorem}\label{th:satnonneg}
Let $(n_v,n_c,m_i,m_{ij})$ be a \sat instance. Then there exists a family of matrices $(\op M_s)_{s\in\set S}$ such that $\exists s:\op M_s$ nonnegative iff the instance is satisfiable.
\end{theorem}

To prove this, we first need the following technical lemma.
\begin{lemma}\label{lem:techC}
Let $(n_v,n_c,m_i,m_{ij})$ be a \sat instance. Then there exists a family of matrices $(\op C_s)_{s\in\set S}$ such that $\exists s:$ the first $n_c$ on-diagonal $4\times4$ blocks of $\op C_s$ are nonnegative iff the instance is satisfiable. In addition, we have $\op C_s^2=\op C_t^2\ \forall s,t$. Furthermore, $(\op C_s)_{s\in\set S}\subset\sqrt{\op C_s^2}$, and the complement contains no nonnegative root.
\end{lemma}
\begin{proof}
For every boolean variable $m_k$, define a vector $v_k\in\sset R^d$ such that their first $n_c$ elements are defined via
\[(v_k)_i:=\begin{cases}1&\text{$m_k$ occurs in $i$\th clause}\\
0&\text{otherwise}\point
\end{cases}
\] We will specify the dimension $d$ later---obviously $d\ge n_c$, and the free entries are used to orthonormalize all vectors in the end. For now, we denote the orthonormalization region with $\vec o$.
We further define the vectors $c_1,c_2\in\sset R^d$ to have all $1$s in the first $n_c$ entries, i.e. $c_{1,2}=(1,\ldots,1,\vec o_{1,2})$.
Let then
\begin{align}
\op C'_s:=c_1c_1\transpose\otimes\begin{pmatrix}
1 & 1 \\ -1 & 1
\end{pmatrix}&+
\frac12c_2c_2\transpose\otimes\begin{pmatrix}
0 & 1 \\ 1 & 0
\end{pmatrix}
+\sum_{k=1}^{n_v} p_kv_kv_k\transpose\otimes\begin{pmatrix}
0 & 1 \\ -1 & 0
\end{pmatrix}\label{eq:cs}\point
\end{align}
The variables $p_k$ denote a specific choice of the rescaled boolean variables $m_i$, which---in order to avoid degeneracy---have to be distinct, i.e via
\begin{equation}\label{eq:booldistinct}
p_i=\left(1-\frac1N-\frac{i}{Nn_v}\right)m_i\quad\forall i=1,\ldots,n_v\point
\end{equation}
The $p_{ij}$ are defined accordingly from the $m_{ij}$ and $N\in\sset N$ is large but fixed.

Let further
\[\op C_s:=\begin{pmatrix}
\op C'_s & 0 \\ 0 & 0
\end{pmatrix}\in\sset C^{d\times d}\comma\]
where we have used an obvious block notation to pad $\op C'_s$ with zeroes, which will come into play later.

The on-diagonal $2\times2$ blocks of $\op C_s$ then encode the \sat inequalities from \prettyref{eq:3sat2}---demanding nonnegativity---as the set of equations
\[
\frac32+p_{i1}+p_{i2}+p_{i3}\ge0
\quad\text{and}\quad
-\frac12-p_{i1}-p_{i2}-p_{i3}\ge0\point
\]
Note that we leave enough head space such that the rescaling in \prettyref{eq:booldistinct} does not affect any of the inequalities---see \prettyref{sec:sing} for details.

Observe further that the eigenvalues corresponding to each eigenprojector in the last term of \prettyref{eq:cs} necessarily have opposite sign, otherwise we create complex entries. We will later rescale $\op C_s$ by a positive factor, which clearly does not affect the inequalities, so the first claim follows.

We can always orthonormalize the vectors ${c_{1,2}}$ and ${v_k}$ using  the freedom left in $\vec o$, hence we can achieve that $\op C_s^2=\op C_t^2\ \forall s,t$.
It is straightforward to check that no other sign choice for the eigenvalues of the first two terms yields nonnegative blocks---see \prettyref{fig:proof2x2} for details. From this, the last two claims follow.
\end{proof}

\begin{figure}[t]
\centering
\includegraphics[width=\textwidth]{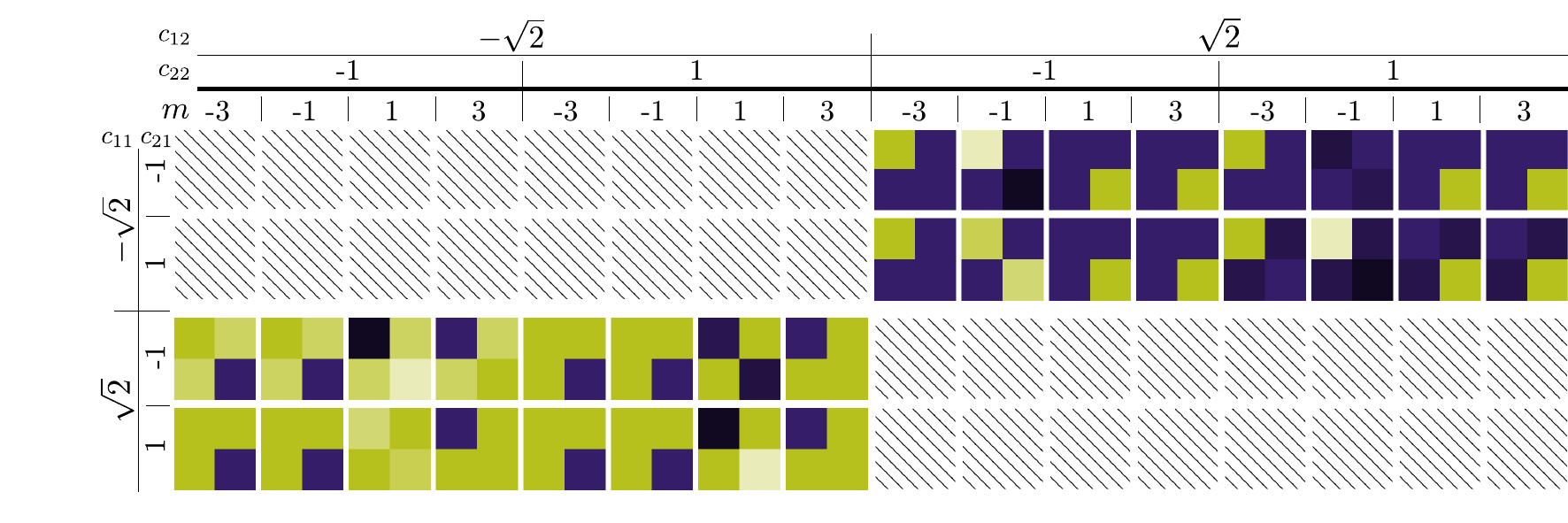}
\caption{$\op C'_s$ for various sign choices of the eigenvalues $c_{ij}$, $i,j=1,2$ corresponding to the eigenvectors ${c_{1,2}}$. Only all positive signs and $m:=\sum_jm_{ij}=-1$ yields a nonnegative block (third from right in top row). Hatching signifies complex numbers, the color scale is the same as in \prettyref{fig:unsatex}.}
\label{fig:proof2x2}
\end{figure}

\subsection{Orthonormalization and Handling the Unwanted Inequalities}
As in \cite{Cubitt2012a}, we have unwanted inequalities---the off-diagonal blocks in the first $4n_c$ entries and the blocks involving the orthonormalization region $\vec o$. We first deal with the off-diagonal blocks in favour of enlarging the orthonormalization region, creating more---potentially negative---entries in there, and then fix the latter.

\paragraph{Off-Diagonal Blocks}
We begin with the following lemma.
\begin{lemma}\label{lem:techE}
Let the family $(\op C_s)_{s\in\set S}$ be defined as in the proof of \prettyref{lem:techC}, and $(n_v,n_c,m_i,m_{ij})$ the corresponding \sat instance. Then there exists a matrix $\op E\in\sset C^{d\times d}$ such that the top left $4n_c\times4n_c$ block of $\op C_s+\op E$ has at least one negative entry $\forall s$ iff the instance is not satisfiable. Furthermore, $\map{im}\op C_s\perp\map{im}\op E\ \forall s$, and $\op C_s+\op E'$ has negative entries $\forall s$, $\forall \op E'\in\sqrt{\op E^2}\setminus\{\op E\}$.
\end{lemma}
\begin{proof}
Define
\[\op E_1:=E_1E_1\transpose\otimes\begin{pmatrix}
1 & 1 & 0 & 0 \\ 1 & 1 & 0 & 0 \\ 1 & 1 & 0 & 0 \\ 1 & 1 & 0 & 0
\end{pmatrix}\quad\text{where}\quad E_1:=(1,\ldots,1,\vec o)\transpose\point\]
Then $\op E_1$ has rank $1$.

From this mask, we now erase the first $n_c$ on-diagonal $4\times4$-blocks, while leaving all other entries in the upper left $4n_c\times4n_c$ block positive. Define $b_i:=(e_i,\vec o)\in\sset C^d$ for $i=1,\ldots,n_c$ where $e_i$ denotes the $i$\th unit vector, and let
\[\op E:=\frac72\op E_1-\frac72\sum_{i=1}^{n_c}t_ib_ib_i\transpose\otimes\begin{pmatrix}
1 & 1 & -1 & 0 \\ 1 & 1 & -1 & 0 \\ 0 & 0 & 0 & 0 \\ 0 & 0 & 0 & 0
\end{pmatrix}\point\]
The variables $t_i$ are chosen close to $1$ but distinct, e.g.
\begin{equation}\label{eq:maskdistinct}
t_i:=\left(1-\frac1M-\frac{i}{Mn_c}\right)\comma
\end{equation}
where $M\in\sset N$ large but fixed. Then $\op E$ has rank $n_c+1$, and adding $\op E$ to $\op C_s$ trivializes all unwanted inequalities in the upper left $4n_c\times4n_c$ block. By picking $M$ large enough, the on-diagonal inequalities are left intact.

One can check that all other possible sign choices for the roots of $\op E$ create negative entries in parts of the upper left block where $\op C_s$ is zero $\forall s$. Furthermore, $\op C_s$ and $\op E$ have distinct nonzero eigenvalues by construction---the orthogonality condition is again straightforward, hence the last two claims follow.
\end{proof}

\paragraph{Orthonormalization Region}
\begin{lemma}\label{lem:techD}
Let $4n<d$ and $\delta\gg1$. There exists a nonnegative rank $2$ matrix $\op D\in\sset C^{d\times d}$ such that the top left $4n\times 4n$ block of $\op D$ has entries $\op D_{ij}\in\mathcal O(\delta^{-2})$ if $j\nmid4$ and the rest of the matrix entries are $\Omega(\delta^{-1})$. If $\op D'\in\sqrt{\op D^2}$, either the same holds true for $\op D'$, or $\op D'_{ij}<0\ \forall j<4n+1,j\mid4$.
\end{lemma}
\begin{proof}
Define
\[E_2:=\bigg(\underbrace{\frac1\delta,\ldots,\frac1\delta}_{n\ \text{times}},1,\ldots,1\bigg)\in\sset C^d\]
and let $\op E_2:=E_2E_2\transpose\otimes\op 1_4$, where $\op 1_4:=(1)_{1\le i,j\le4}$. Let further
\[\Delta:=\bigg(\underbrace{\frac1\delta,\ldots,\frac1\delta}_{n\ \text{times}},-\frac1\delta,\ldots,-\frac1\delta,a\bigg)\in\sset C^d\comma\]
where $0<a<1$ is used to orthonormalize $\Delta$ and ${E_2}$, which is the case if
\[a=-\frac{n}{\delta^2}+\frac{d-n-1}{\delta}\point\]
By explicitly writing out the rank $2$ matrix
\[\op D:=\op E_2\pm\Delta\Delta\transpose\otimes\begin{pmatrix}
1 & 1 & 1 & 0 \\ 1 & 1 & 1 & 0 \\ 1 & 1 & 1 & 0 \\ 1 & 1 & 1 & 0
\end{pmatrix}\comma\]
it is straightforward to check that $\op D$ fulfills all the claims of the lemma---see \prettyref{fig:unsatex} for an example.
\end{proof}


\subsection{Lifting Singularities}\label{sec:sing}
The reader will have noted by now that even though we have orthonormalized all our eigenspaces, ensuring that the nonzero eigenvalues are all distinct, we have at the same time introduced a high-dimensional kernel in $\op C_s$, $\op E$ and $\op D$. The following lemma shows that this does not pose an issue.

\begin{lemma}\label{lem:sing}
  Let $(\op A_s)_{s\in\set S}$ be the family of primary rational roots of some degenerate $\op B\in\sset Q^{d\times d}$. Then there exists a non-degenerate matrix $\op B'$, such that for the family $(\op A'_s)_{s\in\set S}$ of roots of $\op B'$, we have $\op A_s$ positive iff $\op A'_s$ positive. Furthermore, the entries of $\op A'_s$ are rational with bit complexity $\mathfrak r(\op A'_s)=\mathcal O(\map{poly}(\mathfrak r(\op A_s))$.
\end{lemma}
\begin{proof}
  Take a matrix $\op A\in(\op A_s)_{s\in\set S}$. We need to distort the zero eigenvalues $\{\lambda^{(0)}_i\}$ slightly away from $0$. Using notation from \prettyref{def:char}, a conservative estimate for the required smallness without affecting positivity would be
  \[\lambda^{(0)}_i\longmapsto\lambda'^{(0)}_i\quad:\quad0<\lambda'^{(0)}_i\le\frac{1}{C\cdot d^3\cdot\max_{ij}\{|\op Z_{ij}|,|\op Z^{-1}_{ij}|\}}\comma\]
  where we used the Jordan canonical form $\op A=\op Z\op\Lambda\op Z^{-1}$ for some invertible $\op Z$ and $\op\Lambda=\map{diag}(\op J_0,\op J_1)$, such that $\op J_0$ collects all Jordan blocks corresponding to the eigenvalue $0$, and $\op J_1$ collects the remaining ones.
\end{proof}
This will lift all remaining degeneracies and singularities, without affecting our line of argument above. Observe that all inequalities in our construction were bounded away from $0$ with enough head space independent of the problem size, so positivity in the lemma is sufficient.

We thus constructed an embedding of \sat into non-derogatory and non-degenerate matrices, as desired. It is crucial to note that we do not lose anything by restricting the proof to the study of these matrices, as the following lemma shows.
\begin{lemma}\label{lem:nondeg}
There exists a Karp reduction of the \textsc{Divisibility} problems when defined for \emph{all} matrices to the case of \emph{non-degenerate and non-derogatory} matrices.
\end{lemma}
\begin{proof}
As shown in \prettyref{lem:npcontainment}, containment in \np\ for this problem is easy to see, also in the degenerate or derogatory case. Since \sat is \np-complete, there has to exist a poly-time reduction of the \textsc{Divisibility} problems---when defined for \emph{all} matrices---to \sat. Now embed this \sat-instance with our construction. This yields a poly-time reduction to the non-degenerate non-derogatory case.
\end{proof}

\subsection{Complete Embedding}\label{sec:emb}
\begin{figure}[h]
\centering
\subfigure[coding part $\op C_s$]{\includegraphics[width=0.4\textwidth]{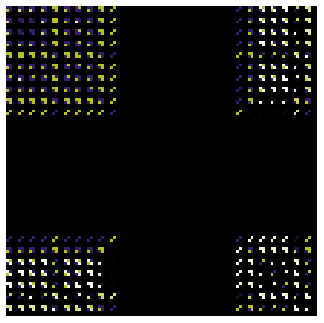}}\hspace{.8em}
\subfigure[mask $\op E$]{\includegraphics[width=0.4\textwidth]{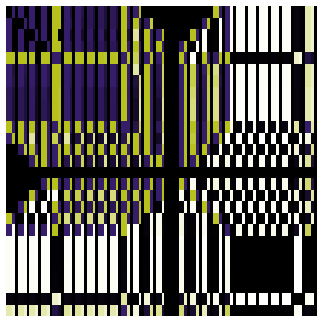}}\\
\subfigure[mask $\op D$]{\includegraphics[width=0.4\textwidth]{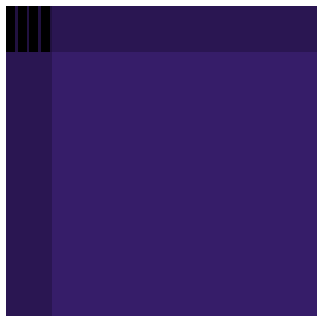}}\hspace{.8em}
\subfigure[combined matrix $\op M_s$]{\includegraphics[width=0.4\textwidth]{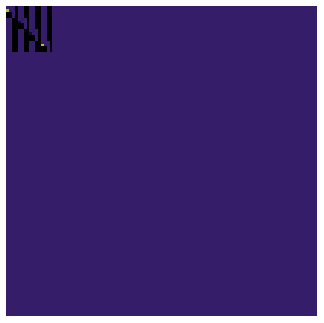}}\\
\vspace{1em}
\hspace{-.25em}\subfigure[color scale, clamped]{\includegraphics[width=0.8\textwidth + 1.7em]{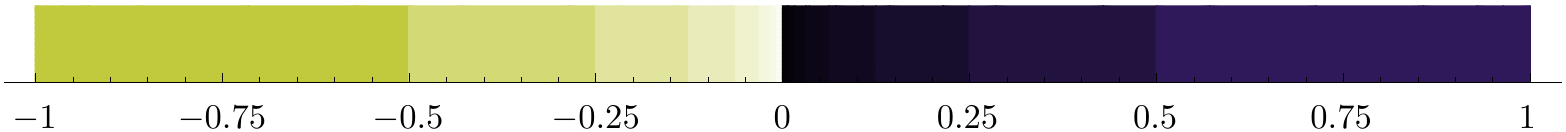}}\\
\subfigure[\sat instance]{
$\map R(m_1,m_2,m_4)\land\map R(m_2,m_3,m_4)\land\map R(m_1,m_3,m_4)\land\map R(m_1,m_2,m_3)$
}
\caption{One branch of an unsatisfiable instance of \sat encoded into a matrix of total rank $19$. The negative entries---two bright dots---in the upper left block in the combined matrix (d) indicate that this branch does not satisfy the given instance. By looking at all other blocks, one sees that none translates to a nonnegative matrix. Observe that in this na\"ive implementation the orthonormalization region is suboptimally large.}
\label{fig:unsatex}
\end{figure}

We now finally come to the proof of \prettyref{th:satnonneg}.
\begin{proof}[\prettyref{th:satnonneg}]
Construct the family $(\op C_s+\op E)_{s\in\set S}$ using \prettyref{lem:techC} and \prettyref{lem:techE}, ensuring that all orthonormalizing is done, which preliminarily fixes the dimension $d$. By \prettyref{lem:techD}, we now construct a mask $\op D(\delta)$ of dimension $d+d'$, where $d'>0$ is picked such that we can also orthonormalize all previous vectors with respect to $ {E_2}$ and $\delta$.

By lemmas \ref{lem:techC}, \ref{lem:techE}, \ref{lem:techD} and \ref{lem:sing}, the perturbed family $(\op M'_s)_{s\in\set S}:=(\op C_s+\op E+N\op D(\delta))'_s$---where $N$ and $\delta\in\sset Q$ are chosen big enough so that all unwanted inequalities are trivially satisfied---fulfils the claims of the theorem and the proof follows.
\end{proof}

We finalize the construction as follows. In \prettyref{th:satnonneg}, we have embedded a given \sat instance into a family of matrices $(\op M_s)_{s\in\set S}$, such that the instance is satisfiable if and only if at least one of those matrices is nonnegative.

By rescaling the entire matrix such that $\max_{ij}(\op M_s)_{ij}=1/2$,
we could show that this instance of \sat is satisfiable if and only if the normalized matrix family $(\op Q_s)_{s\in\set S}$, which we construct explicitly, contains a stochastic matrix.

As shown in \prettyref{sec:equiv}, this can clearly be answered by \textsc{Stochastic Divisibility}, as the family $(\op Q_s)_{s\in\set S}$ comprises all the roots of a unique matrix $\op P$. If this matrix is \emph{not} stochastic, our instance of \sat is trivially not satisfiable. If the matrix \emph{is} stochastic, we ask \textsc{Stochastic Divisibility} for an answer---a positive outcome signifies satisfiability, a negative one non-satisfiability.

\subsection{Bit Complexity of Embedding}
To show that our results holds for only polynomially growing bit complexity, observe the following proposition.
\begin{proposition}
The bit complexity $\mathfrak r(\op M_s)$ of the constructed embedding of a \sat instance $(n_v,n_c,m_i,m_{ij})$ equals $\mathcal O(\map{poly}(n_v,n_c))$.
\end{proposition}
\begin{proof}
We can ignore any construction that multiplies by a constant prefactor, for example \prettyref{lem:nonnegequiv1} and \prettyref{lem:nonnegequiv2}.
The renormalization for \prettyref{lem:nonnegequiv1} to $\max_{ij}\op M_{s,ij}=1/2$ does not affect $\mathfrak r$ either.

The rescaling in \prettyref{eq:booldistinct} and \prettyref{eq:maskdistinct} yields a complexity of $\mathcal O(\log n_v)$, and the same thus holds true for \prettyref{lem:techC} and \prettyref{lem:techE}.

The only other place of concern is the orthonormalization region. Let us write ${a_i}$ for all vectors that need orthonormalization. In the $n$\th step, we need to make up for $\mathcal O(n)$ entries with our orthonormalization, using the same amount of precision to solve the linear equations $({a_i\transpose}{a_n}=0)_{1\le i<n}$. This has to be done with a variant of the standard \knownth{Gauss} algorithm, e.g. the \knownth{Bareiss} algorithm---see for example \cite{Bareiss1968}---which has nonexponential bit complexity.

Together with the lifting of our singularities, which has polynomial precision, we obtain $\mathfrak r(\op M_s)=\mathcal O(\map{poly}(n_v,n_c))$.
Completing the embedding in \prettyref{sec:emb} changes the bit complexity by another polynomial factor, at most, and hence the claim follows.
\end{proof}

\section{Distribution Divisibility}\label{sec:distributions}
\subsection{Introduction}
Underlying stochastic and quantum channel divisibility, and---to some extent---a more fundamental topic, is the question of divisibility and decomposability of probability distributions and random variables. An illustrative example is the distribution of the sum of two rolls of a standard six-sided die, in contrast to the single roll of a twelve-sided die. Whereas in the first case the resulting random variable is obviously the sum of two uniformly distributed random variables on the numbers $\{1,\ldots,6\}$, there is no way to achieve the outcome of the twelve-sided die as any sum of nontrivial ``smaller'' dice---in fact, there is no way of dividing \emph{any} uniformly distributed discrete random variable into the sum of non-constant random variables. In contrast, a uniform continuous distribution can always be decomposed\footnote{All continuous uniform distributions decompose into the sum of a discrete \knownth{Bernoulli} distribution and another continuous uniform distribution. This decomposition is never unique.} into two \emph{different} distributions.

To be more precise, a random variable $X$ is said to be \knownth{divisible} if it can be written as $X=Y+Z$, where $Y$ and $Z$ are non-constant independent random variables that are identically distributed (iid). Analogously, \knownth{infinite divisibility} refers to the case where $X$ can be written as an infinite sum of such iid random variables.

If we relax the condition $Y\overset d=Z$---i.e. we allow $Y$ and $Z$ to have different distributions---we obtain the much weaker notion of \knownth{decomposability}. This includes using other sources of randomness, not necessarily uniformly distributed.

Both divisibility and decomposability have been studied extensively in various branches of probability theory and statistics. Early examples include \knownth{Cramer's theorem} \cite{Cramr1936}, proven in 1936, a result stating that a Gaussian random variable can only be decomposed into random variables which are also normally distributed. A related result on $\chi^2$ distributions by \knownth{Cochran} \cite{Cochran1934}, dating back to 1934, has important implications for the analysis of covariance.

An early overview over divisibility of distributions is given in \cite{Steutel1979}. Important applications of $n$-divisibility---the divisibility into $n$ iid terms---is in modelling, for example of bug populations in entomology \cite{Katti1977}, or in financial aspects of various insurance models \cite{Thorin1977,Thorin1977a}. Both examples study the overall distribution and ask if it is compatible with an underlying subdivision into smaller random events. The authors also give various conditions on distributions to be infinitely divisible, and list numerous infinitely divisible distributions.

Important examples for infinite divisibility include the \knownth{Gaussian}, \knownth{Laplace}, \knownth{Gamma} and \knownth{Cauchy} distributions, and in general all normal distributions. It is clear that those distributions are also finitely divisible, and decomposable. Examples of indecomposable distributions are \knownth{Bernoulli} and discrete uniform distributions.

However, there does not yet exist a straightforward way of checking whether a given discrete distribution is divisible or decomposable. We will show in this work that the question of decomposability is \np-hard, whereas divisibility is in \pp. In the latter case, we outline a computationally efficient algorithm for solving the divisibility question. We extend our results to weak-membership formulations (where the solution is only required to within an error $\epsilon$ in total variation distance), and argue that the continuous case is computationally trivial as the indecomposable distributions form a dense subset.

We start out in \prettyref{sec:prelims2} by introducing general notation and a rigorous formulation of divisibility and decomposability as computational problems. The foundation of all our distribution results is by showing equivalence to polynomial factorization, proven in \prettyref{sec:polyfac}. This will allow us to prove our main divisibility and decomposability results in \prettyref{sec:div} and \ref{sec:dec}, respectively.

\subsection{Preliminaries}\label{sec:prelims2}
\subsubsection{Discrete Distributions}
In our discussion of distribution divisibility and decomposability, we will use the standard notation and language as described in the following definition.
\begin{definition}\label{def:dist}
	Let $(\Omega, \F,  \map p)$ be a discrete \emph{probability space}, i.e. $\Omega$ is at most countably infinite and the \emph{probability mass function} $\map p:\Omega\longrightarrow[0,1]$---or \emph{pmf}, for short---fulfils $\sum_{x\in\Omega}\map p(x)=1$. We take the \emph{$\sigma$-algebra} $\F$ to be maximal, i.e. $\F=2^\Omega$, and without loss of generality assume that the \emph{state space} $\Omega=\sset N$. Denote the \emph{distribution} described by $\map p$ with $\mathfrak D$. A \emph{random variable} $X:\Omega\longrightarrow\set B$ is a measurable function from the sample space to some set $\set B$, where usually $\set B=\sset R$.
\end{definition}

For the sake of completeness, we repeat the following well-known definition of characteristic functions.
\begin{definition}\label{def:char}
	Let $\mathfrak D$ be a discrete probability distribution with pmf $\map p$, and $X\sim\mathfrak D$. Then
	\[ \phi_X(\omega):=\E(\ee^{\ii\omega X})=\int_\Omega\ee^{\ii\omega x}\dd F_X(x)=\sum_{x\in\Omega}\map p(x)\ee^{\ii\omega x} \]
	defines the \emph{characteristic function} of $\mathfrak D$.
\end{definition}
It is well-known that two random variables with the same characteristic function have the same cumulative density function.

\begin{definition}
	Let the notation be as in \prettyref{def:dist}. Then the distribution $\mathfrak D$ is called \emph{finite} if $\map p(k)=0\ \forall k\ge N$ for some $N\in\sset N$.
\end{definition}

\begin{remark}\label{rem:aligned}
	Let $\mathfrak D$ be a discrete probability distribution with pmf $\map p$. We will---without loss of generality---assume that $\map p(0)\ne0$ and $\map p(k)=0\ \forall k<0$ for the pmf $\map p$ of a finite distribution. It is a straightforward shift of the origin that achieves this.
\end{remark}

\subsubsection{Continuous Distributions}\label{sec:nondiscrete}
\begin{definition}\label{def:contdist}
	Let $(\mathcal X,\mathcal A)$ be a \emph{measurable space}, where $\mathcal A$ is the $\sigma$-algebra of $\mathcal X$. The probability distribution of a \emph{random variable} $X$ on $(\mathcal X,\mathcal A)$ is the \knownth{Radon-Nikodym derivative} $f$, which is a measurable function with $\p(X\in\set A)=\int_\set Af\dd\mu$, where $\mu$ is a reference measure on $(\mathcal X,\mathcal A)$.
\end{definition}
Observe that this definition is more general than \prettyref{def:dist}, where the reference measure is simply the counting measure over the discrete sample space $\Omega$. Since we are only interested in real-valued univariate continuous random variables, observe the following important
\begin{remark}
	We restrict ourselves to the case of $\mathcal X=\sset R$ with $\mathcal A$ the \knownth{Borel} sets as measurable subsets and the \knownth{Lebesgue} measure $\mu$. In particular, we only regard distributions with a \emph{probability density function} $f$---or \emph{pdf}, for short---i.e. we require the \emph{cumulative distribution function} $\map P(x):=\p(X\le x)\equiv\int_{y\le x} f(y)\dd y$ to be absolutely continuous.
\end{remark}
\begin{corollary}\label{cor:contdist}
	The cumulative distribution function $\map P$ of a continuous random variable $X$ is almost everywhere differentiable, and any piecewise continuous function $f$ with $\int_\sset Rf(x)\dd x=1$ defines a valid continuous distribution.
\end{corollary}

\subsubsection{Divisibility and Decomposability of Distributions}
To make the terms mentioned in the introduction rigorous, note the two following definitions.
\begin{definition}\label{def:dec}
	Let $X$ be a random variable. It is said to be $n$-\emph{decomposable} if $X=Z_1+\ldots+Z_n$ for some $n\in\sset N$, where $Z_1,\ldots,Z_n$ are independent non-constant random variables. $X$ is said to be \emph{indecomposable} if it is not decomposable.
\end{definition}
\begin{definition}\label{def:div}
	Let $X$ be a random variable. It is said to be $n$-\emph{divisible} if it is $n$-decomposable as $X=\sum_{i=1}^nZ_i$ and $Z_i\overset d=Z_j\ \forall i,j$. $X$ is said to be \emph{infinitely} divisible if $X=\sum_{i=1}^\infty Z_i$, with $Z_i\sim\mathfrak D$ for some nontrivial distribution $\mathfrak D$.
\end{definition}
If we are not interested in the exact number of terms, we also simply speak of \emph{decomposable} and \emph{divisible}. We will show in \prettyref{sec:completedec} that---in contrast to divisibility---the question of decomposability into more than two terms is not well-motivated.

Observe the following extension of \prettyref{rem:aligned}.
\begin{lemma}\label{lem:aligned2}
	Let $\mathfrak D$ be a discrete probability distribution with pmf $\map p$. If $\map p$ obeys \prettyref{rem:aligned}, then we can assume that its factors do as well. In the continuous case, we can without loss of generality assume the same.
\end{lemma}
\begin{proof}
	Obvious from positivity of convolutions in case of divisibility. For decomposability, we can reach this by shifting the terms symmetrically.
\end{proof}

\subsubsection{Markov Chains}\label{sec:mkv}
To establish notation, we briefly state some well-known properties of \knownth{Markov} chains
\begin{remark}\label{rem:mkv}
	Take discrete iid random variables $Y_1,\ldots,Y_n\sim\mathfrak D$ and write $\p(Y_i=k)=p_k:=\map p(k)$ for all $k\in\sset N$, independent of $i=1,\ldots,n$. Define further
	\[X_i:=\begin{cases}
	Y_1+\ldots+Y_i & i>0 \\ 0 &\text{otherwise}\point
	\end{cases}\]
	Then $\{X_n,n\ge0\}$ defines a discrete-time \knownth{Markov chain}, since
	\begin{align*}\p(X_{n+1}=k_{n+1}|X_0=k_0\land\ldots\land X_n=k_n)&=\p(Y_{n+1}=k_{n+1}-k_n)\\
	&\equiv p_{k_{n+1}-k_n}\point
	\end{align*}
\end{remark}
This last property is also called stationary independent increments (iid).

\begin{remark}\label{rem:transmat}
	Let the notation be as in \prettyref{rem:mkv}. The transition probabilities of the Markov chain are then given by
	\[P_{ij}:=\begin{cases}
	p_{j-i} & j\ge i \\ 0&\text{otherwise}\point
	\end{cases}\]
	In matrix form, we write the \emph{transition matrix}
	\[\op P:=\begin{pmatrix}
	p_0 & p_1 & p_2 & \cdots \\
	& p_0 & p_1 & \cdots \\
	& & p_0 & \cdots \\
	& & & \ddots
	\end{pmatrix}\point\]
\end{remark}

Working with transition matrices is straightforward---if the initial distribution is given by $\pi:=(1,0,\ldots)$, then obviously $(\pi\op P)_i=p_i$.
Iterating $\op P$ then yields the distributions of $X_2,X_3,\ldots$, respectively---e.g. $(\pi\op P^2)_i=\p(X_2=i)\equiv\p(Y_1+Y_2=i)$.

We know that $X_2$ is divisible---namely into $X_2=Y_1+Y_2$, by construction---but what if we ask this question the other way round?
We will show in the next section that there exists a relatively straightforward way to calculate if an (infinite) matrix in the shape of $\op P$ has a stochastic root---i.e. if $\mathfrak D$ is divisible. Observe that this is not in contradiction with \prettyref{th:stochhard}, as the theorem does not apply to infinite operators.

In contrast, the more general question of whether we can write a finite discrete random variable as a sum of nontrivial, potentially distinct random variables will be shown to be \np-hard.

\subsection{Equivalence to Polynomial Factorization}\label{sec:polyfac}
Starting from our digression in \prettyref{sec:mkv} and using the same notation, we begin with the following definition.
\begin{definition}\label{def:charpoly}
Denote with $\op S$ the shift matrix $S_{ij}:=\delta_{i+1,j}$. Then we can write
\[\op P=p_0\1+p_1\op S+p_2\op S^2+\ldots=\sum_{i=0}^\infty p_i\op S^i\in\sset R_{[0,1]}[\op S]\point\]
Since $\op S$ just acts as a symbol, we write
\[ f_\mathfrak D(x):=\sum_{i=0}^Np_ix^i\in\mathcal R\quad\text{where}\quad\mathcal R:=\modfrac{\sset R_{\ge0}[x]}{\sim}\comma \]
and $f\sim g:\Leftrightarrow f=cg, c>0$. We call $f_\mathfrak D$ the \emph{characteristic polynomial} of $\mathfrak D$---not to be confused with the characteristic polynomial of a matrix. The equivalence space $\mathcal R$ defines the set of all characteristic polynomials, and can be written as
\[\mathcal R=\bigcup_{i=n}^\infty\mathcal R_i\mt{where}\mathcal R_n:=\modfrac{\mathcal R}{(x^n)}\point\]
\end{definition}
We mod out the overall scaling in order to keep the normalization condition $\sum_k\map p(k)=1$ implicit---if we write $f_\mathfrak D$, we will always assume $f_\mathfrak D(1)=1$. An alternative way to define these characteristic polynomials is via characteristic functions, as given in \prettyref{def:char}.
\begin{definition}\label{def:charpolequiv}
$f_\mathfrak D(\ee^{\ii\omega})=\phi_X(\omega)$.
\end{definition}

The reason for this definition is that it allows us to reduce operations on the transition matrix $\op P$ or products of characteristic functions $\phi_X$ to algebraic operations on $f_\mathfrak D$. This enables us to translate the divisibility problem into a polynomial factorization problem and use algebraic methods to answer it. Because we will make use of it later, we also observe the following.
\begin{definition}\label{def:charpolynorm}
We define norms on the space of characteristic polynomials of degree $N$---$\mathcal R_N$---via $\|f_\mathfrak D\|_{N,p}:=\|(p_i)_{1\le i\le N}\|_{\ell^p}$. If $N$ is not explicitly specified, we usually assume $N=\deg f_\mathfrak D$.
\end{definition}

First note the following proposition.
\begin{proposition}
There is a $1$-to-$1$ correspondence between finite distributions $\mathfrak D$ and characteristic polynomials $f_\mathfrak D$, as defined in \prettyref{def:charpoly}.
\end{proposition}
\begin{proof}
Clear by \prettyref{def:charpolequiv} and the uniqueness of characteristic functions.
\end{proof}
While this might seem obvious, it is worth clarifying, since this correspondence will allow us to directly translate results on polynomials to distributions.

The following lemma reduces the question of divisibility and decomposability---see \prettyref{def:dec} and \ref{def:div}---to polynomial factorization.
\begin{lemma}\label{lem:divtopoly}
  A finite discrete distribution $\mathfrak D$ is $n$-divisible iff there exists a polynomial $g\in\mathcal R$ such that $g^n=f_\mathfrak D$. $\mathfrak D$ is $n$-decomposable iff there exist polynomials $g_1,\ldots,g_n\in\mathcal R$ such that $\prod_{i=1}^ng_i=f_\mathfrak{D}$.
\end{lemma}
\begin{proof}
  Assume that $\mathfrak D$ is $n$-divisible, i.e. that there exists a distribution $\mathfrak D'$ and random variables $Z_1,\ldots,Z_n\sim\mathfrak D'$ such that $X=\sum_{i=1}^nZ_i$. Denote with $\op Q$ the transition matrix of $\mathfrak D'$, as defined in \prettyref{rem:transmat}, and write $\map q$ for its probability mass function.
  Then
  \[\p(X=j)=\p\left(\sum_{i=1}^nZ_i=j\right)=(\op Q^n\pi)_j\comma\]
  as before. Write $g_{\mathfrak D'}$ for the characteristic polynomial of $\mathfrak D'$. By \prettyref{def:charpoly}, $g^n(\op S)\equiv f_\mathfrak D(\op S)$, and hence $g_{\mathfrak D'}^n=f_\mathfrak D$.  Observe that
  \[1=\sum_i \map p(i)=f_\mathfrak D(1)\equiv g_{\mathfrak D'}^n(1)=\left(\sum_i \map q(i)\right)^n\comma\]
  and hence $\sum_i \map q(i)=1$ is normalized automatically.

  The other direction is similar, as well as the case of decomposability, and the claim follows.
\end{proof}

\subsection{Divisibility}\label{sec:div}
\subsubsection{Computational Problems}
We state an exact variant of the computational formulation of the question according to \prettyref{def:div}---i.e. one with an allowed margin of error---as well as a weak membership formulation.
\begin{definition}[\textsc{Distribution Divisibility$_n$}]\label{def:fdddiv}\leavevmode
\begin{description}
\item[Instance.] Finite discrete random variable $X\sim\mathfrak D$.
\item[Question.] Does there exist a finite discrete distribution $\mathfrak D': X=\sum_{i=1}^nZ_i$ for random variables $Z_i\sim\mathfrak D'$?
\end{description}
\end{definition}
Observe that this includes the case $n=2$, which we defined in \prettyref{def:div}.

\begin{definition}[\textsc{Weak Distribution Divisibility}$_{n,\epsilon}$]\label{def:wfdddiv}\leavevmode
\begin{description}
\item[Instance.] Finite discrete random variable $X\sim\mathfrak D$ with pmf $\map p_X(k)$.
\item[Question.] If there exists a finite discrete random variable $Y$ with pmf $\map p_Y(k)$, such that $\|\map p_X-\map p_Y\|_\infty<\epsilon$ and such that
\begin{enumerate}
\item $Y$ is $n$-divisible---return \textsc{Yes}
\item $Y$ is not $n$-divisiblie---return \textsc{No}.
\end{enumerate}
\end{description}
\end{definition}

\subsubsection{Exact Divisibility}
\begin{theorem}\label{th:rootp}
\textsc{Distribution Divisibility}$_n\in\set P$.
\end{theorem}
\begin{proof}
  By \prettyref{lem:divtopoly} it is enough to show that for a characteristic polynomial $f\in\mathcal R_N$, we can find a $g\in\mathcal R:g^n=f$ in polynomial time. In order to achieve this, write $(f)^{1/n}$ as a Taylor expansion with rest, i.e.
  \[
    \sqrt[n]{f(x)}=p(x)+R(x)\mt{where}p\in\mathcal R_{N/n}, R\in\mathcal R_N\point
  \]
  If $R\equiv0$, then $g=p$ $n$-divides $f$, and then the distribution described by $f$ is $n$-divisible. Since the series expansion is constructive and can be done efficiently---see \cite{Muller1987}---the claim follows.

  If the distribution coefficients are rational numbers, another method is to completely factorize the polynomial---e.g. using the \knownth{LLL} algorithm, which is known to be easy in this setting---sort and recombine the linear factors, which is also in $\map O(\map{poly}(\map{ord} f))$, see for example \cite{Hart2011}. Then check if all the polynomial root coefficients are positive.
\end{proof}

We collect some further facts before we move on.
\begin{remark}
Let $\map p$ be the probability mass function for a finite discrete distribution $\mathfrak D$, and write $\supp\map p=\{k:\map p(k)\ne0\}$. If $\max\supp\map p-\min\supp\map p=:w$, then $\mathfrak D$ is obviously not $n$-divisible for  $n>w/2$, and furthermore not for any $n$ that do not divide $w,n<w/2$. Indeed, $\mathfrak D$ is not $n$-divisible if the latter condition holds for \emph{either} $\max\supp\map p$ or $\min\supp\map p$.
\end{remark}
\begin{remark}
Let $X\sim\mathfrak D$ be an $n$-divisible random variable, i.e. $\exists Z_1,\ldots,Z_n\sim\mathfrak D':\sum_{i=1}^n Z_i=X$. Then $\mathfrak D'$ is unique.
\end{remark}
\begin{proof}
This is clear, because $\sset R[x]$ is a unique factorization domain.
\end{proof}

\subsubsection{Divisibility with Variation}
As an intermediate step, we need to extend \prettyref{th:rootp} to allow for a margin of error $\epsilon$, as captured by the following definition.
\begin{definition}[\textsc{Distribution Divisibility}$_{n,\epsilon}$]\label{def:fdddive}\leavevmode
	\begin{description}
		\item[Instance.] Finite discrete random variable $X\sim\mathfrak D$ with pmf $\map p_X(k)$.
		\item[Question.] Do there exist finite random variables $Z_1,\ldots, Z_n\sim\mathfrak D'$ with pmfs $\map p_Z(k)$, such that $\|\underbrace{\map p_Z\ast\ldots\ast\map p_Z}_\text{$n$ times}-\map p_X\|_\infty<\epsilon$?
	\end{description}
\end{definition}

\begin{lemma}\label{lem:fddwd}
\textsc{Distribution Divisibility}$_{n,\epsilon}$ is in \pp.
\end{lemma}
\begin{proof}
Let $f(x)=\sum_{i=0}^Np_ix^i$ be the characteristic polynomial of a finite discrete distribution, and $\epsilon>0$. By padding the distribution with $0$s, we can assume without loss of generality that $N=\deg f$ is a multiple of $n$.
A polynomial root---if it exists---has the form $g(x)=\sum_{i=0}^Na_ix^i$, where $a_i\ge0\ \forall i$. Then
\begin{align*}
g(x)^n&=(\ldots+a_3 x^3+a_2 x^2+a_1 x+a_0)^n\\
&=\ldots+((n-1)a_1^2+n a_0^{n-2} a_2) x^2+n a_0^{n-1} a_1 x+a_0^n\point
\end{align*}
Comparing coefficients in the divisibility condition $f(x)=g(x)^n$, the latter translates to the set of inequalities
\begin{align*}
a_0^n&\in(p_0-\epsilon,p_0+\epsilon)\\
n a_0^{n-1} a_1&\in(p_1-\epsilon,p_1+\epsilon)\\
(n-1)a_1^2+n a_0^{n-2} a_2&\in(p_2-\epsilon,p_2+\epsilon)\\
&\ \,\vdots
\end{align*}
Each term but the first one is of the form $\map h_i(a_1,\ldots,a_{i-1})+na_0^{n-i}a_i\in(p_i-\epsilon,p_i+\epsilon)$, where $\map h_i\ge0\ \forall i$ is monotonic. This can be rewritten as $a_i\in\set U_{\epsilon/na_0^{n-i}}((p-\map h_i(a_1,\ldots,a_{i-1}))/na_0^{n-i})$. It is now easy to solve the system iteratively, keeping track of the allowed intervals $\set I_i$ for the $a_i$.

If $\set I_i=\emptyset$ for some $i$, we return \textsc{No}, otherwise \textsc{Yes}. We have thus developed an efficient algorithm to answer \textsc{Distribution Weak Divisibility}$_{n,\epsilon}$, and the claim of \prettyref{lem:fddwd} follows.
\end{proof}

\begin{remark}
	Given a random variable $X$, the algorithm constructed in the proof of \prettyref{lem:fddwd} allows us to calculate the closest $n$-divisible distribution to $X$ in polynomial time.
\end{remark}
\begin{proof}
	Straightforward, e.g. by using binary search over $\epsilon$.
\end{proof}

\subsubsection{Weak Divisibility}\label{sec:weakdiv}
For the weak membership problem, we reduce \textsc{Weak Distribution Divisibility}$_{n,\epsilon}$ to \textsc{Distribution Divisibility}$_{n,\epsilon}$.

\begin{theorem}\label{th:weakdiv}
\textsc{Weak Distribution Divisibility}$_{n,\epsilon}\in\pp$.
\end{theorem}
\begin{proof}
  Let $\mathfrak D$ be a finite discrete distribution. If \textsc{Distribution Divisibility}$_{n,\epsilon}$ answers \textsc{Yes}, we know that there exists an $n$-divisible distribution $\epsilon$-close to $\mathfrak D$. In case of \textsc{No}, $\mathfrak D$ itself is not $n$-divisible, hence we know that there exists a non-$n$-divisible distribution close to $\mathfrak D$.
\end{proof}

\subsubsection{Continuous Distributions}\label{sec:contdiv}
Let us briefly discuss the case of continuous distributions---continuous meaning a non-discrete state space $\mathcal X$, as specified in \prettyref{sec:nondiscrete}.
Although divisibility of continuous distributions is well-defined and widely studied, formatting the continuous case as a computational problem is delicate, as the continuous distribution must be specified by a finite amount of data for the question to be computationally meaningful. The most natural formulation is the continuous analogue of \prettyref{def:div} as a weak-membership problem. However, we can show that this problem is computationally trivial.

First observe the following intermediate result.
\begin{lemma}\label{lem:densediv}
	Take $f\in\Ccb^+$ with $\supp f\subset\set A\cup\set B$, where $\set A:=[0,M], \set B:=[2M,3M]$, $M\in\sset R_{>0}$. We claim that if $f$ is divisible, then both $f|_\set A$ and $f|_\set B$ are divisible.
\end{lemma}
\begin{proof}
	   Due to symmetry, it is enough to show divisibility for $f|_\set A$.
	   Assuming $f$ is divisible, we can write $f=r\ast r$, i.e. $f(x)=\int_\sset Rr(x-y)r(y)\dd y$. It is straightforward to show that $r(x)=0\ \forall x<0$. Define
	   \begin{equation}\label{eq:densediv1}
	   \bar r(x)=\begin{cases}
	   r(x) & x\in\set A/2\\
	   0 & \text{otherwise}\comma
	   \end{cases}
	   \end{equation}
	   where $\set A/2:=\{a/2:a\in\set A\}$. Then
	   \begin{align*}
	   	(\bar r\ast\bar r)(x)&=\int_\sset R\bar r(x-y)\bar r(y)\dd y\\
	   	&=\int_\sset R\dd y\begin{cases}
	   		r(x-y) & x-y\in\set A/2\\0&\text{otherwise}
	   	\end{cases}\cdot\begin{cases}
	   	r(y)&y\in\set A/2\\0&\text{otherwise}\point
	   \end{cases}
	\end{align*}
	We see that $(\bar r\ast\bar r)(x)=0$ for $x\not\in\set A$. For $x\in\set A$, the support of the integrand is contained in $\{y:y\in x-\set A/2\land y\in\set A/2\}=x-\set A/2\cap\set A/2:=\set S_x$, and hence we can write $(\bar r\ast\bar r)(x)=\int_{\set S_x}r(x-y)r(y)\dd y$.
	It hence remains to show that $f|_\set A(x)=\int_{\set S_x}r(x-y)r(y)\dd y\ \forall x\in\set A$. The integrand $r(x-y)r(y)=0\ \forall y<0\lor y>x$. The difference in the integration domains can be seen in \prettyref{fig:intdom}. We get two cases.
	\begin{figure}[t]
		\centering
		\subfigure[]{
			\begin{tikzpicture}[scale=.9]
			\draw[color=gray!50] (0,0) grid[xstep=2, ystep=2] (4.3,4.3);
			\draw[fill=myGreen!80] (0,0) -- (4,4) -- (4,0) -- cycle;
			\draw[fill=myPurple] (0,0) -- (2,2) -- (4,2) -- (2,0) -- cycle;
			\draw [<->,very thick] (0,4.5) node (yaxis) [above] {$y$} |- (4.5,0) node (xaxis) [right] {$x$};
			\draw (0,-.3) node{0};
			\draw (2,-.3) node{M/2};
			\draw (4,-.3) node{M};
			\draw (0,0) node[left]{0};
			\draw (0,2) node[left]{M/2};
			\draw (0,4) node[left]{M};
			\end{tikzpicture}
		}
		\qquad
		\subfigure[]{
			\begin{tikzpicture}[scale=.9]
			\draw[color=gray!50] (0,0) grid[xstep=2, ystep=2] (4.3,4.3);
			\draw[fill=myGreen!80] (0,0) -- (4,4) -- (4,0) -- cycle;
			\draw[fill=myPurple] (0,0) -- (3.2,3.2) -- (4,3.2) -- (0.8,0) -- cycle;
			\draw [<->,very thick] (0,4.5) node (yaxis) [above] {$y$} |- (4.5,0) node (xaxis) [right] {$x$};
			\draw (0,-.3) node{0};
			\draw (2,-.3) node{M/2};
			\draw (4,-.3) node{M};
			\draw (0,0) node[left]{0};
			\draw (0,2) node[left]{M/2};
			\draw (0,4) node[left]{M};
			\end{tikzpicture}
		}
		\caption{(a) Integration domains in \prettyref{prop:densediv} for $\bar r\ast\bar r$ (purple) and $f|_\set A$ (light green), respectively. (b) Example for integration domains in \prettyref{prop:densedec} for $\bar r\ast\bar s$ (purple) and $f|_\set A$ (light green), respectively. }
		\label{fig:intdom}
	\end{figure}

	Let $x\in\set A$. Assume $\exists y'\in(M/2,M)$ such that $r(x-y')r(y')>0$. Let $x':=2y'$. We then have $r(y')^2=r(x'-y')r(y')>0$, and due to continuity $f(x')>0$, contradiction, because $x'\in(M,2M)$.

	Analogously fix $x'\in(M/2,M)$. Assume $\exists y'\in(0,x'-M/2)$ such that $r(x'-y')r(y')>0$, and thus $r(x'-y')>0$, where $a:=x'-y'>M/2$, $2a\in(M,2M)$. Then $r(a)^2=r(2a-a)r(a)>0$, due to continuity $f(2a)>0$, again contradiction.
\end{proof}

\begin{proposition}\label{prop:densediv}
  Let $\Ccb^+$ denote the set of piecewise continuous nonnegative functions of bounded support. Then the set of \emph{nondivisible} functions, $\mathcal I:=\{f:\nexists r\in\Ccb: f=r\ast r\}$ is dense in $\Ccb$.
\end{proposition}
\begin{proof}
It is enough to show the claim for functions $f\in\Ccb^+$ with $\inf\supp f\ge0$. Let $\epsilon>0$, and $M:=\sup\supp f$. Take $j\in\Ccb$ to be nondivisible with $\supp j\subset(2M,3M)$, and define
\[
g(x):=\begin{cases}
f(x)&x<M\\ \epsilon j(x)/\|j\|_\infty&x\in(2M,3M) \\ 0 & \text{otherwise}\point
\end{cases}
\]
By construction, $\|f-g\|_\infty<\epsilon$, but $g|_{(2M,3M)}\equiv j$ is not divisible, hence by \prettyref{lem:densediv} $g$ is not divisible, and the claim follows.
\end{proof}

\begin{corollary}\label{cor:densediv}
  Let $\epsilon>0$. Let $X$ be a continuous random variable with pdf $\map p_X(k)$. Then there exists a nondivisible random variable $Y$ with pdf $\map p_Y(k)$, such that $\|\map p_X-\map p_Y\|<\epsilon$.
\end{corollary}
\begin{proof}
  Let $\epsilon>0$ small. Since $\Ccb\subset\{f\ \text{integrable}\}=:\set L$, we can pick $f_M\in\set L:\supp f_M\in(-M,M),\|\map p_X-f_M\|<\epsilon/3$ and $\|f_M\|=1+\delta$ with $|\delta|\le\epsilon/3$. Then
  \[
  \left\lVert\map p_X-\frac{f_M}{\|f_M\|}\right\rVert=\left\lVert\map p_X-\frac{f_M}{1+\delta}\right\rVert\le\|\map p_X-f_M\|+\frac{\epsilon}{2}\|f_M\|\le\epsilon\point
  \]
  and \prettyref{prop:densediv} finishes the claim.
\end{proof}

\begin{corollary}
  Any weak membership formulation of divisibility in the continuous setting is trivial to answer, as for all $\epsilon>0$, there always exists a nondivisible distribution $\epsilon$ close to the one at hand. Similar considerations apply to other formulations of the continuous divisibility problem.
\end{corollary}

\subsubsection{Infinite Divisibility}
Let us finally and briefly discuss the case of infinite divisibility. While interesting from a mathematical point of view, the question of infinite divisibility is ill-posed computationally. Trivially, discrete distributions cannot be infinitely divisible, as follows directly from \prettyref{th:rootp}. A similar argument shows that neither the $\epsilon$, nor the weak variant of the discrete problem is a useful question to ask, as can be seen from \prettyref{lem:fddwd} and \ref{th:weakdiv}.

By the same arguments as in \prettyref{sec:contdiv}, the weak membership version is easy to answer and thus trivially in \pp.

\subsection{Decomposability}\label{sec:dec}
\subsubsection{Computational Problems}
We define the decomposability analogue of \prettyref{def:fdddiv} and \ref{def:wfdddiv} as follows.
\begin{definition}[\textsc{Distribution Decomposability}]\label{def:fdddec}\leavevmode
\begin{description}
\item[Instance.] Finite discrete random variable $X\sim\mathfrak D$.
\item[Question.] Do there exist finite discrete distributions $\mathfrak D', \mathfrak D'': X=Z_1+Z_2$ for random variables $Z_1\sim\mathfrak D', Z_2\sim\mathfrak D''$?
\end{description}
\end{definition}

\begin{definition}[\textsc{Weak Distribution Decomposability}$_\epsilon$]\label{def:wfdddec}\leavevmode
\begin{description}
\item[Instance.] Finite discrete random variable $X\sim\mathfrak D$ with pmf $\map p_X(k)$.
\item[Question.] If there exists a finite discrete random variable $Y$ with pmf $\map p_Y(k)$, such that $\|\map p_X-\map p_Y\|_\infty<\epsilon$ and such that
\begin{enumerate}
\item $Y$ is decomposable---return \textsc{Yes}
\item $Y$ is indecomposable---return \textsc{No}.
\end{enumerate}
\end{description}
\end{definition}

In this section, we will show that \textsc{Distribution Decomposability} is \np-hard, for which we will need a series of intermediate results.
Requiring the support of the first random variable $Z_1$ to have a certain size, i.e. $|\supp(p_\mathfrak{D'})|=m$, yields the following program.
\begin{definition}[\textsc{Distribution Decomposability}$_m$, $m\ge2$]\leavevmode
\begin{description}
\item[Instance.] Finite discrete random variable $X\sim\mathfrak D$ with $|\supp(p_\mathfrak{D'})|>m$.
\item[Question.] Do there exist finite discrete distributions $\mathfrak D', \mathfrak D'': X=Z_1+Z_2$ for random variables $Z_1\sim\mathfrak D', Z_2\sim\mathfrak D''$ and such that $|\supp(p_\mathfrak{D'})|=m$?
\end{description}
\end{definition}
We then define \textsc{Distribution Even Decomposability} to be the case where the two factors have equal support.

The full reduction tree can be seen in \prettyref{fig:reduction2}.
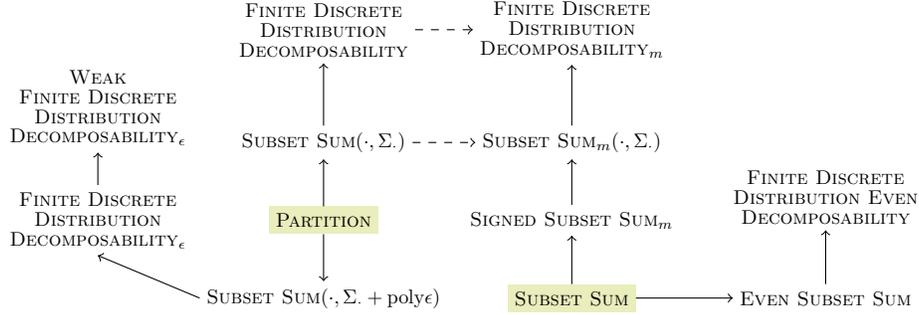
\begin{figure}
\centering
\begin{tikzpicture}[node distance=1.38cm,auto,every node/.style={scale=0.75}]
  \node (part) [fill=myGreenL] {\textsc{Partition}};
  \node (sspe) [below of=part] {\textsc{Subset Sum}$(\cdot,\Sigma_\cdot+\map{poly}\epsilon)$};
  \node (fddde) [left of=part,node distance=4cm] {\textsc{\shortstack{Finite Discrete\\Distribution\\Decomposability$_\epsilon$}}};
  \node (wfddde) [above of=fddde,node distance=2cm] {\textsc{\shortstack{Weak\\Finite Discrete\\Distribution\\Decomposability$_\epsilon$}}};

  \node (ssp) [above of=part] {\textsc{Subset Sum}$(\cdot,\Sigma_\cdot)$};
  \node (fdddiv) [above of=ssp,node distance=2cm] {\textsc{\shortstack{Finite Discrete\\Distribution\\Decomposability}}};
  \node (sssm) [right of=part,node distance=4.4cm] {\textsc{Signed Subset Sum}$_m$};
  \node (ss) [below of=sssm,fill=myGreenL] {\textsc{Subset Sum}};
  \node (sspm) [above of=sssm] {\textsc{Subset Sum}$_m(\cdot,\Sigma_\cdot)$};
  \node (fdddivm) [above of=sspm,node distance=2cm] {\textsc{\shortstack{Finite Discrete\\Distribution\\Decomposability$_m$}}};
  \node (evenss) [right of=ss,node distance=4.5cm] {\textsc{Even Subset Sum}};
  \node (evenfdddiv) [above of=evenss,node distance=1.8cm] {\textsc{\shortstack{Finite Discrete\\Distribution Even\\Decomposability}}};
  \draw[->] (part) to node {} (ssp);
  \draw[->] (part) to node {} (sspe);
  \draw[->] (sspe.west) to node {} (fddde.south);
  \draw[->] (fddde) to node {} (wfddde);

  \draw[->] (ssp) to node {} (fdddiv);
  \draw[->,dashed] (fdddiv.east) to node {} (fdddivm.west);
  \draw[->,dashed] (ssp.east) to node {} (sspm.west);
  \draw[->] (ss) to node {} (sssm);
  \draw[->] (sssm) to node {} (sspm);
  \draw[->] (sspm)  to node {} (fdddivm);
  \draw[->] (ss.east) to node {} (evenss.west);
  \draw[->] (evenss) to node {} (evenfdddiv);
\end{tikzpicture}
\caption{Complete chain of reduction for our discrete programs. The dashed lines are obvious and not mentioned explicitly.}
\label{fig:reduction2}
\end{figure}

Analogous to \prettyref{lem:npcontainment}, we state the following observation.
\begin{lemma}\label{lem:npcontainment2}
	All the above \textsc{Decomposability} problems in \prettyref{def:fdddec} to \ref{def:wfdddec} are contained in \np.
\end{lemma}
\begin{proof}
	It is straightforward to construct a witness and a verifier that satisfies the definition of the decision class \np. For example in \prettyref{def:fdddece}, a witness is given by two tables of numbers which are easily checked to form finite discrete distributions. Convolving these lists and comparing the result to the given distribution can clearly be done in polynomial time. Both verification and witness are thus poly-sized, and the claim follows.
\end{proof}

\subsubsection{Even Decomposability}\label{sec:even}
We continue by proving that \textsc{Distribution Even Decomposability} is \np-hard. We will make use of the following variant of the well-known \textsc{Subset Sum} problem, which is \np-hard---see \prettyref{lem:evenoptpart} for a proof. The interested reader will find a rigorous digression in \prettyref{sec:nptools}.
\begin{definition}[\textsc{Even Subset Sum}]\label{def:evenss}\leavevmode
  \begin{description}
    \item[Instance.] Multiset $\set S$ of reals with $|\set S|$ even, $l\in\sset R$.
    \item[Question.] Does there exist a multiset $\set T\subsetneq \set S$ with $|\set T|=|\set S|/2$ and such that
    $|\sum_{t\in\set T}t-\sum_{s\in\set S\setminus\set T}s|<l$?
  \end{description}
\end{definition}
This immediately leads us to the following intermediate result.
\begin{lemma}\label{lem:even}
	\textsc{Distribution Even Decomposability} is \np-hard.
\end{lemma}
\begin{proof}
	Let $(\set S,l)$ be an instance of \textsc{Even Subset Sum}. We will show that there exists a polynomial $f\in\mathcal R$ of degree $2|\set S|$ such that $f$ is divisible into $f=g\cdot h$ with $\deg g=\deg h$ iff $(\set S,l)$ is a \yes instance.
	We will explicitly construct the polynomial $f\in\mathcal R$. As a first step, we transform the \textsc{Even Subset Sum} instance $(\set S,l)$, making it suited for embedding into $f$.

	Let $N:=|\set S|$ and denote the elements in $\set S$ with $s_1,\ldots,s_N$. We perform a linear transformation on the elements $s_i$ via
	\begin{equation}\label{eq:bi}
	b_i:=a\left(s_i-\frac1{|\set S'|}\sum_{s\in\set S'}s\right)+\frac{al}{2|\set S'|}\mt{for} i=1,\ldots,N\comma
	\end{equation}
	where $a\in\sset R_{>0}$ is a free scaling parameter chosen later such that $|b_i|<\delta\in\sset R_+$ small. Let $\set B:=\{b_1,\ldots,b_N\}$. By \prettyref{lem:rescale}, we see that \textsc{Even Subset Sum}$(\set S,l)=$ \textsc{Even Subset Sum}$(\set B,al)$. Since further $\sum_ib_i=al/2>0$, we know that $(\set B,al)$ is a \yes instance if and only if there exist two non-empty disjoint subsets $\set B_1\cup\set B_2=\set B$ such that both
	\begin{equation}\label{eq:subsetineq}
	\sum_{i\in\set B_1}b_i>0\mt{and}\sum_{i\in\set B_2}b_i>0\point
	\end{equation}

	The next step is to construct the polynomial $f$ and prove that it is divisible into two polynomial factors $f=g\cdot h$ if and only if $(\set B,al)$ is a \yes instance. We first define quadratic polynomials $g(b_i,x):=x^2+b_ix+1$ for $i=1,\ldots,N$, and set $f_\set T(x):=\prod_{b\in\set T}g(b,x)$ for $\set T\subset\set B$. Observe that for suitably small $\delta$, the $g(b_i,x)$ are irreducible over $\sset R[x]$. With this notation, we claim that $f_\set B(x)$ has the required properties.

	In order to prove this claim, we first show that for sufficiently small scaling parameter $a$, a generic subset $\set T\subset\set B$ with $n:=|\set T|$ and $f_\set T(x)=:\sum_{i=1}^{2|\set T|}c_ix^i$, the coefficients $c_i$ satisfy
	\begin{align}
		c_0&=1\comma\\
		\map{sgn}(c_1)&=\map{sgn}(\Sigma)\comma\\
		c_{2j}&>0\mt{for}j=1,\ldots,|\set T|\comma\label{eq:ineqeven}\\
		\map{sgn}(c_{2j+1})&\ge\map{sgn}(\Sigma)\mt{for}j=1,\ldots,|\set T|-1\comma\label{eq:ineqodd}
	\end{align}
	where $\Sigma:=\sum_{t\in\set T}t$. Indeed, if then $f_\set B=g\cdot h$, where $g,h\in\mathcal R$, then $g=f_{\set B_1}$ and $h=f_{\set B_2}$ for aforementioned subsets $\set B_1,\set B_2\subsetneq\set B$, and conversely if $(\set B,al)$ is a \yes instance, then $f_\set B=f_{\set B_1}\cdot f_{\set B_2}$---remember that $\sset R[x]$ is a unique factorization domain, so all polynomials of the shape $f_\set T$ necessarily decompose into quadratic factors.

	By construction, $c_0=1$ and $c_1=n\Sigma$, so the first two assertions follow immediately.
	To address  \prettyref{eq:ineqeven} and \ref{eq:ineqodd}, we further split up the even and odd coefficients into
	\begin{equation}\label{eq:coeffs}
		c_j=:\begin{cases}
		c_{j,0}+c_{j,2}+\ldots+c_{j,j}&\text{if $j$ even}\\
		c_{j,1}+c_{j,3}+\ldots+c_{j,j}&\text{if $j$ odd}\comma
		\end{cases}
	\end{equation}
	where $c_{j,k}$ is the coefficient of $x^jb_{i_1}\cdots b_{i_k}$. We thus have $c_{j,k}=\mathcal O(\delta^k)$ in the limit $\delta\rightarrow0$---we will implicitly assume the limit in this proof and drop it for brevity.	 Our goal is to show that the scaling in $\delta$ suppresses the combinatorial factors, i.e. that $c_j$ is dominated by its first terms $c_{j,0}$ and $c_{j,1}$, respectively.

	In order to achieve this, we need some more machinery. First regard $g(\delta,x)=x^2+\delta x+1$. It is imminent that for an expansion
  \[
  g(\delta,x)^n=:\sum_{j=0}^{2n}x^j\sum_{k=0}^n d_{j,k}\delta^k\comma
  \]
  we get coefficient-wise inequalities
  \begin{equation}\label{eq:coefficientwise}
	|c_{j,k}|\le d_{j,k}\ \forall j=0,\ldots,2n, k=0,\ldots,n\point
	\end{equation}
	We will calculate the coefficients $d_{j,k}$ of $g(\delta,x)^n$ explicitly and use them to bound the coefficients $c_{j,k}$ of $f_\set T(x)$.

  Using a standard \knownth{Cauchy} summation and the uniqueness of polynomial functions, we obtain
  \begin{align*}
    g(\delta,x)^n&=\sum_{j=0}^n\frac{1}{j!}(1+x^2)^{n-j}x^j(n)_j\delta^j\\
    &=\sum_{j=0}^n\frac{\delta^j}{j!}(n)_jx^j\sum_{k=0}^{n-j}{n-j\choose k}x^{2k}\\
    &\equiv\sum_{j=0}^\infty\sum_{k=0}^\infty\frac{\delta^j}{j!}(n)_j{n-j\choose k}x^{j+2k}\\
    &=\sum_{j=0}^\infty\sum_{l=0}^j\frac{\delta^l}{l!}(n)_l{n-l\choose j-l}x^{2j-l}\\
    &\equiv\sum_{j=0}^n\sum_{l=0}^j\frac{\delta^l}{l!}(n)_l{n-l\choose j-l}x^{2j-l}\\
    &=\sum_{j=0}^n\sum_{l=j}^{2j}\frac{\delta^{2j-l}}{(2j-l)!}(n)_{2j-l}{n-2j+l\choose l-j}x^l\point
  \end{align*}
  With $(n)_l$ we denote  the falling factorial, i.e. $(n)_l=n(n-1)(n-2)\cdots(n-l+1)$. By convention, $(n)_0=1$.

  Regarding even and odd powers of $x$ separately, we can thus deduce that
  \begin{align*}
  g(\delta,x)^n&=\sum_{j=0}^{2n}x^j\begin{dcases}
  \sum_{k=0}^{\lfloor\frac j2\rfloor}\frac{\delta^{2k+1}}{(2k+1)!}\frac{(n)_{\lceil\frac j2\rceil+k}}{(\lfloor\frac j2\rfloor-k)!} &\text{if $j$ odd}\\
  \sum_{k=0}^{\frac j2}\frac{\delta^{2k}}{(2k)!}\frac{(n)_{\frac j2+k}}{(\frac j2-k)!} &\text{if $j$ even}
  \end{dcases}\\
   &=\sum_{j=0}^{2n}x^j\begin{dcases}
   (n)_{\lceil\frac j2\rceil}\sum_{k=0}^{\lfloor\frac j2\rfloor}\frac{\delta^{2k+1}}{(2k+1)!}\frac{(n-\lceil\frac j2\rceil)_{k}}{(\lfloor\frac j2\rfloor-k)!} &\text{if $j$ odd}\\
   (n)_{\frac j2}\sum_{k=0}^{\frac j2}\frac{\delta^{2k}}{(2k)!}\frac{(n-\frac j2)_{k}}{(\frac j2-k)!} &\text{if $j$ even}\point
   \end{dcases}
  \end{align*}
    A straightforward estimate shows that for the even and odd case, we obtain the coefficient scaling
   \[
   g(\delta,x)^n=\sum_{j=0}^{2n}x^j\begin{dcases}
   (n)_{\lceil\frac j2\rceil}\sum_{k=0}^{\lfloor\frac j2\rfloor}\delta^{2k+1}\mathcal O(n^k) &\text{if $j$ odd}\\
   (n)_{\frac j2}\sum_{k=0}^{\frac j2}\delta^{2k}\mathcal O(n^k) &\text{if $j$ even}\comma
   \end{dcases}
   \]
   which means that e.g. picking $\delta=\mathcal O(1/n^2)$ is enough to exponentially suppress the higher order combinatorial factors.

  	We will now separately address the even and odd case---\prettyref{eq:ineqeven} and \ref{eq:ineqodd}.
  	\paragraph{Even Case} As the constant coefficients $c_{j,0}=\mathcal O(1)$ in $\delta$, it is the same as for $g(\delta,x)^n$ and by \prettyref{eq:coefficientwise}, we immediately get
  	\[
  	\frac{|c_{j,2}+\ldots+c_{j,j}|}{c_{j,0}}=\mathcal O(\delta)\point
  	\]

  	\paragraph{Odd Case}  Note that if $\Sigma<0$, we are done, so assume $\Sigma>0$ in the following. A simple combinatorial argument gives
  	\[
  	c_{j,1}={n-1\choose(j-1)/2}\Sigma\comma
  	\]
  	so it remains to show that $c_{j,1}>-c_{j,3}-\ldots-c_{j,j}$. Analogously to the even case, by \prettyref{eq:coefficientwise}, we conclude
  	\[
  	\frac{|c_{j,3}+\ldots+c_{j,j}|}{c_{j,1}}=\mathcal O(\delta)\comma
  	\]
  	which finalizes our proof.
\end{proof}

\subsubsection{m-Support Decomposability}\label{sec:msupp}
In the next two sections we will generalize the last result to \textsc{Distribution Decomposability}$_m$. As a first observation, we note the following.
\begin{lemma}\label{lem:msupp_p}
Let $f(n)$ be such that $(f(n)\beta(f(n),n+1-f(n)))^{-1}=\mathcal O(\map{poly}(n))$. Then \textsc{Distribution Decomposability}$_{f(|\Cdot|)}\in\pp$.
\end{lemma}
\begin{proof}
See proof of \prettyref{th:rootp}, and an easy scaling argument for ${n\choose f(n)}$ completes the proof. As in \prettyref{rem:ss_m}, this symmetrically extends to \textsc{Distribution Decomposability}$_{|\Cdot|-f(|\Cdot|)}\in\pp$.
\end{proof}
Observe that $f(n)=n/2$ yields exponential growth, hence the remark is consistent with the findings in \prettyref{sec:even}.

We now regard the general case. As in the last section, we need variants of the \textsc{Subset Sum} problem, which are given in the following two definitions.
\begin{definition}[\textsc{Subset Sum}$_m$, $m\in\sset Z$]\label{def:mss}\leavevmode
  \begin{description}
    \item[Instance.] Multiset $\set S$ of reals with $|\set S|$ even, $l\in\sset R$.
    \item[Question.] Does there exist a multiset $\set T\subsetneq \set S$ with $|\set T|=m$ and such that
    $|\sum_{t\in\set T}t-\sum_{s\in\set S\setminus\set T}s|<l$?
  \end{description}
\end{definition}
\begin{definition}[\textsc{Signed Subset Sum}$_m$]\label{def:sssm}\leavevmode
  \begin{description}
    \item[Instance.] Multiset $\set S$ of positive integers or reals, $x,y\in\sset R:x\le y$.
    \item[Question.] Does there exist a multiset $\set T\subset \set S$ with $|\set T|=m$ and such that
    $x<\sum_{t\in\set T}t-\sum_{s\in\set S\setminus\set T}s<y$?
  \end{description}
\end{definition}
Both are shown to be \np-hard in \prettyref{lem:ssm} and \ref{lem:sssm}, or by the following observation.
In order to avoid having to take absolute values in the definition of \textsc{Subset Sum}$_m$, we reduce it to multiple instances of \textsc{Signed Subset Sum}$_m$, by using the following interval partition of the entire range $(-l, l)$.
\begin{remark}\label{rem:subdiv}
For every $a>0, l>0$, there exists a partition of the interval $(-l-2a,l+2a)=\bigcup_{i=0}^{N-1}(x_i,x_{i+1})$ with suitable $N\in\sset N$ such that $x_{i+1}-x_i=2a$ and
\[(-l,l)=\left(\bigcup_{i=1}^{N-2}(x_i,x_{i+1})\right)\setminus\left((x_0,x_1)\cup(x_{N-1},x_N)\right)\point\]
\end{remark}

This finally leads us to the following result.
\begin{lemma}\label{lem:mdiv}
\textsc{Distribution Decomposability}$_m$ is \np-hard.
\end{lemma}
\begin{proof}
We will show the reduction \textsc{Distribution Decomposability}$_m\longleftarrow$ \textsc{Subset Sum}$_m$. Let $m$ be fixed. Let $(\set S,l)$ be an \textsc{Subset Sum} instance. For brevity, we write $\Sigma_\set S:=\sum_{s\in\set S}s$. Without loss of generality, by \prettyref{cor:optpartlin}, we again assume $\Sigma_\set S\ge0$.
Now define $a:=2 (|\set S|l + 2 m\Sigma_\set S - |\set S| \Sigma_\set S)/(2 m - |\set S|)$.
Using \prettyref{rem:subdiv}, pick a suitable subdivision of the interval $(-l-2a,l+2a)$, such that
\begin{align*}
\textsc{Subset Sum}_m(\set S,l) =&\left(\bigvee\limits_{i=1}^{N-2}\textsc{Signed Subset Sum}_m(\set S,x_i,x_{i+1})\right)\\
&\land\lnot\textsc{Signed Subset Sum}_m(\set S,x_0,x_1)\\
&\land\lnot\textsc{Signed Subset Sum}_m(\set S,x_{N-1},x_N)\point
\end{align*}
One can verify that
\begin{align*}
\textsc{Signed}\ &\textsc{Subset Sum}_m(\set S,x_i-a,x_i+a)\\
&=\textsc{Signed Subset Sum}_m(\set S+c(m,i),-\Sigma_{\set S+c(m,i)},\Sigma_{\set S+c(m,i)})\\
&=\textsc{Subset Sum}_m(\set S+c(m,i),\Sigma_{\set S+c(m,i)})\comma
\end{align*}
where we chose $c(m,i)=x_i/(2m-|\set S|)$. The latter program we can answer using the same argument as for the proof of \prettyref{lem:even}, and the claim follows.
\end{proof}
As a side remark, this also confirms the following well-known fact.
\begin{corollary}
Let $f(n)$ be as in \prettyref{lem:msupp_p}. Then \textsc{Subset Sum}$_{f(|\Cdot|)}\in\pp$.
\end{corollary}

\subsubsection{General Decomposability}
We have already invented all the necessary machinery to answer the general case.
\begin{theorem}\label{th:dec2}
\textsc{Distribution Decomposability} is \np-hard.
\end{theorem}
\begin{proof}
Follows immediately from \prettyref{lem:even}, where we regard the special set of \textsc{Subset Sum} instances for which $(\set S,l)$ is such that $l=\sum_{s\in\set S}s$. We show in \prettyref{lem:optpartsum} that \textsc{Subset Sum}$(\Cdot,\Sigma_{\CDot})$ is still \np-hard, thus the claim follows.
\end{proof}

\subsubsection{Decomposability with Variation}
As a further intermediate result---and analogously to \prettyref{def:fdddive}---we need to allow for a margin of error $\epsilon$.
\begin{definition}[\textsc{Distribution Decomposability}$_\epsilon$]\label{def:fdddece}\leavevmode
	\begin{description}
		\item[Instance.] Finite discrete random variable $X\sim\mathfrak D$ with pmf $\map p_X(k)$.
		\item[Question.] Do there exist finite discrete random variables $Z_1\sim\mathfrak D', Z_2\sim\mathfrak D''$ with pmfs $\map p_{Z_1}(k)$, $\map p_{Z_2}(k)$, such that $\|\map p_{Z_1}\ast\map p_{Z_2}-\map p_X\|_\infty<\epsilon$?
	\end{description}
\end{definition}
This definition leads us to the following result.
\begin{lemma}\label{lem:fddde}
\textsc{Distribution Decomposability}$_\epsilon$ is \np-hard.
\end{lemma}
\begin{proof}
First observe that we can restate this problem in the following equivalent form. Given a finite discrete distribution $\mathfrak D$ with characteristic polynomial $f_\mathfrak D$, do there exist two finite discrete distributions $\mathfrak D', \mathfrak D''$ with characteristic polynomials $f_\mathfrak{D'},f_\mathfrak{D''}$ such that $\|f_\mathfrak D-f_\mathfrak{D'} f_\mathfrak{D''}\|_d<\epsilon$? Here, we are using the maximum norm from \prettyref{def:charpolynorm}, and assume without loss of generality that $\deg f_\mathfrak D=\deg f_\mathfrak{D'}\deg f_\mathfrak{D''}$.

As $f_\mathfrak D$ is a polynomial, we can regard its \knownth{Viète} map $v:\sset C^n\longrightarrow\sset C^n$, where $n=\deg f_\mathfrak D$, which continuously maps the polynomial roots to its coefficients. It is a well-known fact---see \cite{Whitney1972} for a standard reference---that $v$ induces an isomorphism of algebraic varieties $w:\sset A_k^n/\set S_n\isomap\sset A_k^n$, where $\set S_n$ is the $n\th$ symmetric group. This shows that $w^{-1}$ is polynomial, and hence the roots of $f_\mathfrak{D'} f_\mathfrak{D''}$ lie in an $\mathcal O(\epsilon)$-ball around those of $f_\mathfrak D$. By a standard uniqueness argument we thus know that if $f_\mathfrak D=\prod_i f_i$ with $f_i=x^2+b_ix+1$ as in the proof of \prettyref{lem:even}, then $f_\mathfrak{D'}=\prod_i g_i$ with $g_i=a_i x^2+b'_ix+c_i$, where $a_i=c_i=1+\mathcal O(\epsilon)$, $b'_i=b_i+\mathcal O(\epsilon)$---we again implicitly assume the limit $\epsilon\rightarrow0$.

We continue by proving the reduction \textsc{Distribution Divisibility}$_\epsilon\longleftarrow$ \textsc{Subset Sum}$(\Cdot,\Sigma_{\CDot}+\map{poly}\,\epsilon)$, which is \np-hard as shown in \prettyref{lem:optpartsumex}. Let $\set S=\{s_i\}_{i=1}^N$ be a \textsc{Subset Sum} multiset. We claim that it is satisfiable if and only if the generated characteristic function $f_\set S(x)$---where we used the notation of the proof of \prettyref{lem:even}---defines a finite discrete probability distribution and the corresponding random variable $X$ is a \textsc{Yes} instance for \textsc{Distribution Divisibility}$_\epsilon$.

First assume $f_\set S$ is such a \textsc{Yes} instance. Then $\sum_{s\in\set S}s\ge0$, and there exist two characteristic polynomials $g=\prod_i g_i$ and $h=\prod_i h_i$ as above and such that $\|f_\set S-gh\|_d<\epsilon$. We also know that if $g_i=a_ix^2+b_ix+c_i$, then $\exists\set T\subsetneq\set S$ such that $\{b_i\}_i\in\set B_{\epsilon}(\set T)\subseteq\sset R^{|\set T|}$, where $\set T\subsetneq\set S$ and $\set B_\epsilon(\set T)$ denotes an $\epsilon$ ball around the set $\set T$, and analogously for $h_i=a'_ix^2+b'_ix+c'_i$, with $\{b'_i\}_i\in\set B_{\epsilon}(\set S\setminus\set T)\subseteq\sset R^{|\set S|-|\set T|}$. Regarding the linear coefficients, we thus have
\begin{align}
\begin{split}\label{eq:fddweq}
\left|\sum_{s\in\set S}s-\sum_{t\in\set T}t-\sum_{s\in\set S\setminus\set T}s\right|&=\left|\sum_{s\in\set S}s-\sum_{i=1}^{|\set T|} b_i-\sum_{i=1}^{|\set S\setminus\set T|} b'_i+\mathcal O(\epsilon)\right|\\
&\le\mathcal O(\epsilon)\le\sum_{s\in\set S}s+\mathcal O(\epsilon)\point
\end{split}
\end{align}

Now the case if $f_\set S$ is a \textsc{No} instance. Assume there exists a nontrivial multiset $\set T\subsetneq\set S$ satisfying
\[
\left|\sum_{t\in\set T}t-\sum_{s\in\set S\setminus\set T}s\right|<\sum_{s\in\set S}s+\mathcal O(\epsilon)\point
\]
Then by construction $\sum_{t\in\set T}t,\sum_{s\in\set S\setminus T}s\ge-\mathcal O(\epsilon)$ and $f_\set T\cdot f_{\set S\setminus\set T}=f_\set S$, contradiction, and the claim follows.
\end{proof}

\subsubsection{Weak Decomposability}
Analogously to \prettyref{sec:weakdiv}, we now regard the weak membership problem of decomposability.
\begin{theorem}
\textsc{Weak Distribution Decomposability}$_\epsilon$ is \np-hard.
\end{theorem}
\begin{proof}
In order to show the claim, we prove the reduction \textsc{Weak Distribution Decomposability}$_\epsilon\longleftarrow$\textsc{Distribution Decomposability}$_{\map g(\epsilon)}$, where the function $\map g=\mathcal O(\epsilon)$. It is clear that the polynomial factor leaves the \np-hardness of the latter program intact.

We use the same notation as in the proof of \prettyref{lem:fddde}. Let $f_\set S$ be a \textsc{Yes} instance of \textsc{Distribution Decomposability}$_\epsilon$, and define $\set S':=\{s+\mathcal O(\epsilon): s\in\set S\}$. From \prettyref{eq:fddweq} it immediately follows that then $f_{\set S'}$ is a \textsc{Yes} instance of \textsc{Distribution Decomposability}$_{\map g(\epsilon)}$, where we allow $\map g=\mathcal O(\epsilon)$. We have hence shown that there exists an $\mathcal O(\epsilon)$ ball around each \textsc{Yes} instance that \emph{solely} contains \textsc{Yes} instances.

A similar argument holds for the \textsc{No} instances. It is clear that these cases can be answered using \textsc{Weak Distribution Decomposability}$_\epsilon$, and the claim follows.
\end{proof}

\subsubsection{Complete Decomposability}\label{sec:completedec}
Another interesting question to ask is for the complete decomposition of a  finite distribution $\mathfrak D$ into a sum of indecomposable distributions. We argue that this decomposition is not unique.

\begin{proposition}\label{prop:counterex}
There exists a family of finite distributions $(\mathfrak D_n)_{n\in\sset N}$ with probability mass functions $\map p_n(k):\max\supp\map p_n(k)=4n$ and such that, for each $\mathfrak D_n$, there are at least $n!$ distinct decompositions into indecomposable distributions.
\end{proposition}
\begin{proof}
\begin{figure}[t]
\centering
\subfigure[complex roots of characteristic polynomial of $\mathfrak G_3$]{\includegraphics[height=6cm]{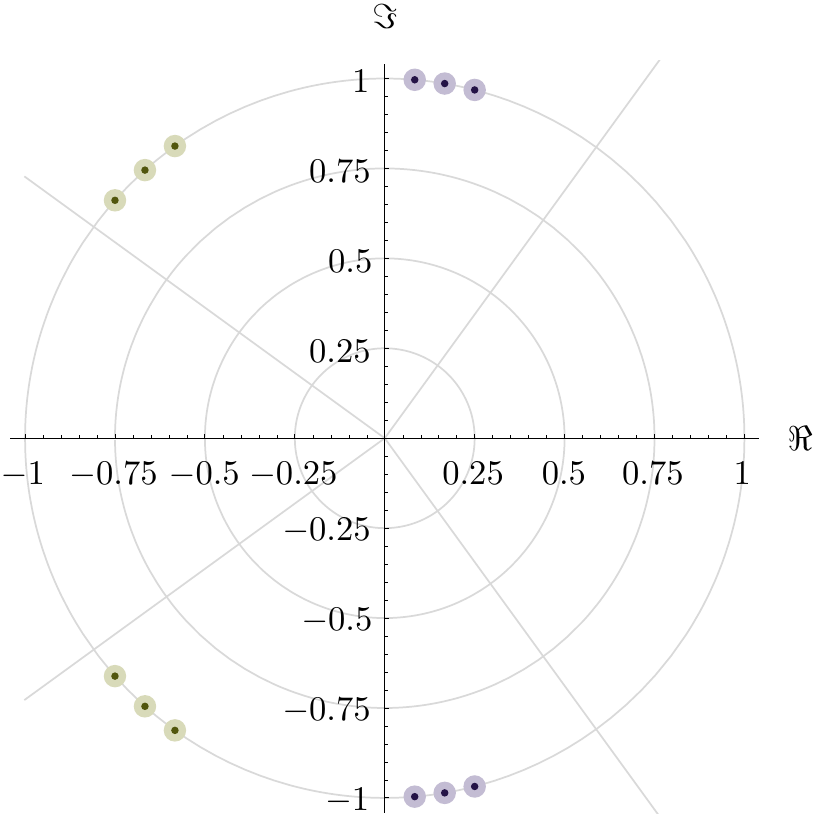}}\hspace{.8em}
\subfigure[probability mass function for $\mathfrak G_3$]{\includegraphics[height=3.6cm]{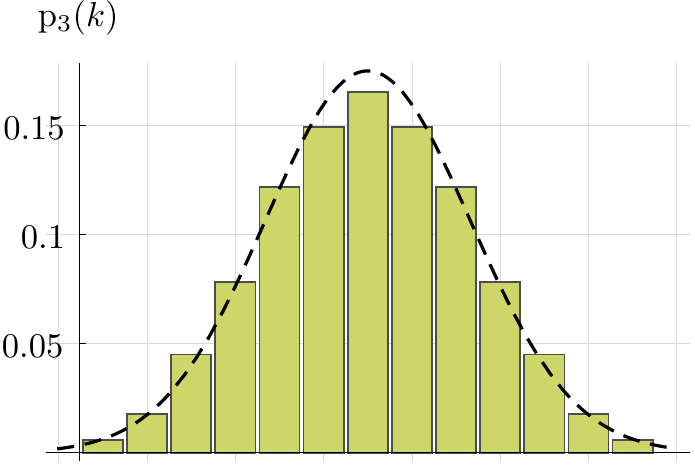}}
\caption{Counterexample construction of \prettyref{prop:counterex}. The dashed line shows a normal distribution for comparison.}
\label{fig:counterex}
\end{figure}
We explicitly construct the family $(\mathfrak D_n)_{n\in\sset N}$. Let $n\in\sset N$. We will define a set of irreducible quadratic polynomials $\{p_k,n_k\ \text{for}\ k=1,\ldots,n\}$ such that $n_k$ are \emph{not} positive, but $p_kn_l$ are positive quartics $\forall k,l$---and thus define valid probability distributions. Since $\sset R[x]$ is a unique factorization domain the claim then follows.

Following the findings in the proof of \prettyref{lem:even}, it is in fact enough to construct a set $\{a_k,b_k:0<|a_k|<2,-2<b_k<0\ \text{for}\ k=1,\ldots,n\}\subset\sset R^{2n}$ and such that $a_k+b_l>0\ \forall k,l$---then let $p_k:=1+a_kx+x^2$, $n_k:=1+b_kx+x^2$. It is straightforward to verify that e.g.
\[
	a_k:=1+\frac{k}{2n}\mt{and}b_k:=-\frac{k}{2n}
\]
fulfil these properties.
\end{proof}
\begin{remark}\label{rem:counterex}
	Observe that for $b_k:=-k/2n^2$, the construction in \prettyref{prop:counterex} allows decompositions into $m$ indecomposable terms, where $m=n,\ldots,2n$.
\end{remark}
\begin{corollary}
	$\mathcal R$ is not a unique factorization domain.
\end{corollary}
Proposition \ref{prop:counterex} and \prettyref{rem:counterex} show that an exponential number of complete decompositions---all of which have different distributions---do not give any further insight into the distribution of interest--indeed, as the number of positive indecomposable factors is not even unique, asking for a non-maximal decomposition into indecomposable terms does not answer more than whether the distribution is decomposable at all.

Indeed, the question whether one \emph{can} decompose a distribution into indecomposable parts can be trivially answered with \yes, but if we include the condition that the factors have to be non-trivial, or for decomposability into a certain number of terms---say $N\ge2$ or the maximum number of terms---the problem is also obviously \np-hard by the previous results.

In short, by \prettyref{th:dec2}, we immediately obtain the following result.
\begin{corollary}
Let $\mathfrak D$ be a finite discrete distribution. Deciding whether one can write $\mathfrak D$ as any nontrivial sum of irreducible distributions is \np-hard.
\end{corollary}

\subsubsection{Continuous Distributions}\label{sec:contdec}
Analogous to our discussion in \prettyref{sec:contdiv}, the exact and $\epsilon$ variants of the decomposability question are computationally ill-posed. We again point out that answering the weak membership version is trivial, since the set of indecomposable distributions is dense, as the following proposition shows.

\begin{proposition}\label{prop:densedec}
	Let $\Ccb^+$ denote the piecewise linear nonnegative functions of bounded support. Then the set of \emph{indecomposable} functions, $\mathcal J:=\{f:\nexists r,s\in\Ccb: f=r\ast s\}$ is dense in $\Ccb$.
\end{proposition}
\begin{proof}
	We first extend \prettyref{lem:densediv}, and again take $f\in\Ccb^+:\supp f\subset\set A\cup\set B$. While not automatically true that $r(x), s(x)=0\ \forall x<0$, we can assume this by shifting $r$ and $s$ symmetrically. We also assume $\inf\supp f=0$, and hence $\inf\supp r=\inf\supp s=0$---see \prettyref{lem:aligned2} for details.

	Since $f(x)=0\ \forall x\in(M, 2M)$, we immediately get $r(x)=s(x)=0\ \forall x\in(M,2M)$. Furthermore, $\exists m\in(0,M):r(x)=s(y)=0\ \forall x\in(m,M],y\in(M-m,M]$. Analogously to \prettyref{eq:densediv1}, we define
	\begin{equation}\label{eq:densedec1}
	\bar r(x)=\begin{cases}
	r(x) & x\in\set [0,m]\\
	0 & \text{otherwise}
	\end{cases}
	\mt{and}
	\bar s(x)=\begin{cases}
	r(x) & x\in\set [0,M-m]\\
	0 & \text{otherwise}\point
	\end{cases}
	\end{equation}
	The integration domain difference is derived analogously, and can be seen in an example in \prettyref{fig:intdom}. We again regard the two cases separately.

	Let $x\in\set A$. Assume $\exists y'\in(M-m,M)$ such that $r(x-y')s(y')>0$. Then $s(y')>0$, contradiction.
	Now fix $x'\in(m,M)$, and assume $\exists y'\in(0,x'-m):r(x'-y')s(y')>0$. Since $x'-y'>x'-x'+m=m$, $r(x'-y')>0$ yields another contradiction.

	The rest of the proof goes through analogously.
\end{proof}

\begin{corollary}\label{cor:densedec}
	Let $\epsilon>0$. Let $X$ be a continuous random variable with pmf $\map p_X(k)$. Then there exists a indecomposable random variable $Y$ with pmf $\map p_Y(k)$, such that $\|\map p_X-\map p_Y\|<\epsilon$.
\end{corollary}
\begin{proof}
	See \prettyref{cor:densediv}.
\end{proof}

\section{Conclusion}
In \prettyref{sec:cptpstochastic}, we have shown that the question of existence of a stochastic root for a given stochastic matrix is in general at least as hard as answering \sat, i.e.\ it is \np-hard. By \prettyref{cor:doubly}, this \np-hardness result also extends to \textsc{Nonnegative} and \textsc{Doubly Stochastic Divisibility}, which proves \prettyref{th:stochhard}.
A similar reduction goes through for \textsc{\cptp Divisibility} in \prettyref{cor:cptpequiv1}, proving \np-hardness of the question of existence of a \cptp root for a given \cptp map.

In \prettyref{sec:distributions}, we have shown that---in contrast to \cptp and stochastic matrix divisibility---distribution divisibility is in \pp, proving \prettyref{th:div}. On the other hand, if we relax divisibility to the more general decomposability problem, it becomes \np-hard as shown in \prettyref{th:dec}. We have also extended these results to weak membership formulations in \prettyref{th:divweak} and \ref{th:decweak}---i.e.\ where we only require a solution to within $\epsilon$ in the appropriate metric---showing that all the complexity results are robust to perturbation.

Finally, in \prettyref{sec:contdiv} and \ref{sec:contdec}, we point out that for continuous distributions---where the only computationally the only meaningful formulations are the weak membership problems or closely related variants---questions of divisibility and decomposability become computationally trivial, as the nondivisible and indecomposable distributions independently form dense sets.

As containment in \np\ for all of the \np-hard problems is easy to show (\prettyref{lem:npcontainment} and \ref{lem:npcontainment2}), these problems are also \np-complete. Thus our results imply that, apart for the distribution divisibility problem which is efficiently solvable, all other divisibility problems for maps and distributions are equivalent to the famous $\pp=\np$ conjecture, in the following precise sense: A polynomial-time algorithm for answering any one of these questions---\textsc{(Doubly) Stochastic}, \textsc{Nonnegative} or \textsc{\cptp Divisibility}, or either of the \textsc{Decomposability} variants---would prove  $\pp=\np$.  Conversely, solving  $\pp=\np$ would imply that there exists a polynomial-time algorithm to solve all of these \textsc{Divisibility} problems.

\section{Acknowledgements}
Johannes Bausch would like to thank the German National Academic Foundation and EPSRC for financial support. Toby Cubitt is supported by the Royal Society. The authors are grateful to the Isaac Newton Institute for Mathematical Sciences, where part of this work was carried out, for their hospitality during the 2013 programme ``Mathematical Challenges in Quantum Information Theory''.

\section{Appendix}
\subsection{\np-Toolbox}\label{sec:nptools}
\paragraph{Boolean Satisfiability Problems}
\begin{definition}[\sat]\label{def:sat}\leavevmode\\
\textbf{Instance:} $n_v$ boolean variables $m_1,\ldots,m_{n_v}$ and $n_c$ clauses $\map R(m_{i1},m_{i2},m_{i3})$ where $i=1,\ldots,n_c$, usually denoted as a $4$-tuple $(n_v,n_c,m_i,m_{ij})$. The boolean operator $\map R$ satisfies
\[
\map R(a,b,c)=\begin{cases}
\textsc{True}&\text{if exactly one of $a,b$ or $c$ is \textsc{True}}\\
\textsc{False}&\text{otherwise}\point
\end{cases}
\]
\textbf{Question:} Does there exist a truth assignment to the boolean variables such that every clause contains exactly one true variable?
\end{definition}

\paragraph{Subset Sum Problems}
We start out with the following variant of a well-known \np-complete problem---see for example \cite{Garey1979} for a reference.


\begin{definition}[\textsc{Subset Sum, Variant}]\leavevmode
\begin{description}
\item[Instance.] Multiset $\set S$ of integer or rational numbers, $l\in\sset R$.
\item[Question.] Does there exist a multiset $\set T\subsetneq \set S$ such that
$|\sum_{t\in\set T}t-\sum_{s\in\set S\setminus\set T}s|<l$?
\end{description}
\end{definition}

From the definition, we immediately observe the following rescaling property.
\begin{corollary}\label{cor:optpartlin}
Let $a\in\sset R\setminus\{0\}$ and $(\set S,l)$ a \textsc{Subset Sum} instance. Then \textsc{Subset Sum}$(\set S,l)=$ \textsc{Subset Sum}$(a\set S,|a|l)$.
\end{corollary}

For a special version of \textsc{Subset Sum}, \textsc{Subset Sum}$_m$---as defined in \prettyref{def:mss}---we observe the following.
\begin{lemma}\label{lem:ssm}
  \textsc{Subset Sum}$_m\longleftarrow$\textsc{Subset Sum}.
\end{lemma}
\begin{proof}
  If $(\set S,l)$ is a \textsc{Subset Sum} instance, then
  \[
    \textsc{Subset Sum}(\set S,l)=\bigvee\limits_{m=1}^{|\set S|}\textsc{Subset Sum}_m(\set S,l)\point\qedhere
  \]
\end{proof}
\begin{remark}\label{rem:ss_m}
It is clear that \textsc{Subset Sum}$_m(\set S,l)=$ \textsc{False} for $|\set S|\le m\lor0\ge m$. Furthermore, \textsc{Subset Sum}$_m(\set S,l)=$ \textsc{Subset Sum}$_{|\set S|-m}(\set S,l)$.
\end{remark}
Observe that this remark indeed makes sense, as \textsc{Subset Sum}$_0$ should give \textsc{False}, which is the desired outcome for $m=|\set S|$.
We further reduce \textsc{Subset Sum} to \textsc{Even Subset Sum}, as defined in \prettyref{def:evenss}.
\begin{lemma}\label{lem:evenoptpart}
\textsc{Even Subset Sum}$\longleftarrow$\textsc{Subset Sum}.
\end{lemma}
\begin{proof}
Let $(\set S, l)$ be an \textsc{Subset Sum} instance. Define $\set S':=\set S\cup\{0,\ldots,0\}:|\set S'|=2|\set S|$. Then if \textsc{Even Subset Sum}$(\set S',l)=$ \textsc{True}, we know that there exists $\set T'\subset\set S':|\sum_{t\in\set T'}t-\sum_{s\in\set S'\setminus\set T'}s|<l$. Let then $\set T:=\set T'$ without the $0$s. It is obvious that then $|\sum_{t\in\set T}t-\sum_{s\in\set S\setminus\set T}s|<l$. The \textsc{False} case reduces analogously, hence the claim follows.
\end{proof}

For \textsc{Even Subset Sum}, we generalize \prettyref{cor:optpartlin} to the following scaling property.
\begin{lemma}\label{lem:rescale}
Let $a\in\sset R\setminus\{0\}$, $c\in\sset R$, and $(\set S,l)$ an \textsc{Even Subset Sum} instance. Then
\textsc{Even Subset Sum}$(\set S,l)=$ \textsc{Even Subset Sum}$(a\set S+c,|a|l)$, where addition and multiplication is defined element-wise.
\end{lemma}
\begin{proof}
Straightforward, since we require $|\set S|=2|\set T|=2|\set S\setminus\set T|$.
\end{proof}

For \prettyref{def:sssm}, we finally show
\begin{lemma}\label{lem:sssm}
\textsc{Signed Subset Sum}$_m\longleftarrow$\textsc{Subset Sum}$_m$.
\end{lemma}
\begin{proof}
Immediate from \textsc{Signed Subset Sum}$_m(\set S,-l,l)=$ \textsc{Subset Sum}$_m(\set S,l)$.
\end{proof}

\paragraph{Partition Problems}
Another well-known \np-complete problem which will come into play in the proof of \prettyref{th:dec2} is \knownth{set partitioning}.
\begin{definition}[\textsc{Partition}]\leavevmode
\begin{description}
\item[Instance.] Multiset $\set A$ of positive integers or reals.
\item[Question.] Does there exist a multiset $\set T\subsetneq \set A$ with $\sum_{t\in\set T}t=\sum_{s\in\set A\setminus T}s$?
\end{description}
\end{definition}

\begin{lemma}\label{lem:optpartsum}
For the special case of \textsc{Subset Sum} with instance $(\set S,l)$, where the bound $l\in\sset R$ equals the total sum of the instance numbers $l=\sum_{s\in\set S}s$, we obtain the equivalence \textsc{Subset Sum}$(\Cdot,\Sigma_{\CDot})\longleftrightarrow$\textsc{Partition}$(\Cdot)$.
\end{lemma}
\begin{proof}
Let $\set S$ be the multiset of a \textsc{Subset Sum} instance $(\set S,l)$, where we assume without loss of generality that all $\set S\ni s\ge0$.
Now first assume $\Sigma_\set S=0$. In that case the claim follows immediately, since the problems are identical.

Without loss of generality, we can thus assume $\Sigma_\set S>0$ and regard the set $\set S':=\set S\cup\{-\Sigma_\set S/2,-\Sigma_\set S/2\}$, such that $\Sigma_{\set S'}=0$.

If now \textsc{Subset Sum}$(\set S',0)=$ \textsc{True}, we know that there exists $\set T\subsetneq\set S':|\sum_{t\in\set T}t-\sum_{s\in\set S'\setminus T}s|=0$. Now assume $\set T$ contains both copies of $-\Sigma_\set S/2$. Then clearly
\[ |\sum_{t\in\set T}t\ -\cl\sum_{s\in\set S'\setminus T}\cl s|=|-\Sigma_\set S+\cll\sum_{t\in\set T\setminus\{-\Sigma_\set S\}}\cll t\ -\sum_{s\in\set S\setminus T}\! s|>0\comma\]
since $|\set S\setminus\set T|>0$.
The same argument shows that exactly one $-\Sigma_\set S/2\in\set T, \set S\setminus\set T$, and hence \textsc{Partition}$(\set S)=$ \textsc{True}.

On the other hand, if \textsc{Partition}$(\set S)=$ \textsc{True}, then it immediately follows that  \textsc{Subset Sum}$(\set S,\Sigma_\set S)=$ \textsc{True}.
\end{proof}

Finally observe the following extension of \prettyref{lem:optpartsum}.
\begin{lemma}\label{lem:optpartsumex}
 Let $\epsilon>0$, $f$ a polynomial. Then \textsc{Subset Sum}$(\Cdot,\Sigma_{\CDot}+f(\epsilon))\longleftrightarrow$ \textsc{Partition}$(\Cdot)$.
\end{lemma}
\begin{proof}
The proof is the same as for \prettyref{lem:optpartsum}, but we regard $\set S\cup\{-\Sigma_\set S/2-f(\epsilon)/2,-\Sigma_\set S/2-f(\epsilon)/2\}$ instead.
\end{proof}

\section{Literature}

\bibliographystyle{elsarticle-harv}
\def\url#1{}
\def\urlprefix{\vspace{-20pt}}
\bibliography{books,distributions,divisibility}

\end{document}